\documentclass[10pt,a4paper,leqno]{article}
\usepackage[latin,english]{babel}
\usepackage[utf8]{inputenc}
\usepackage[T1]{fontenc}
\usepackage{eufrak, color}
\usepackage{geometry}
\usepackage{amsmath, stmaryrd }
\usepackage{amssymb}
\usepackage{amsthm}
\usepackage{makeidx}
\usepackage{mathtools}
\usepackage{mathabx}
\usepackage{cases}
\usepackage{braket}
\usepackage[toc,page]{appendix}
\usepackage{pdfsync}
\usepackage{hyperref}
\usepackage{graphicx}
\usepackage{psfrag}
\usepackage{graphicx}
\usepackage{bbm}
\usepackage{fancyhdr}
\usepackage{math}
\numberwithin{equation}{section}

\theoremstyle{plain}
\newtheorem{thm}{Theorem}[section]
\newtheorem{lem}[thm]{Lemma}
\newtheorem{prop}[thm]{Proposition}
\newtheorem{cor}[thm]{Corollary}
\theoremstyle{definition}
\newtheorem{defn}[thm]{Definition}

\theoremstyle{remark}
\newtheorem{rem}[thm]{Remark}

\newcommand{\LieTr}[2]{e^{\im#1} #2 e^{-\im#1}}

\DeclareMathOperator{\diag}{diag}

\DeclareMathOperator{\meas}{meas}

\newcommand{\lip}{{{\rm Lip}}}
\newcommand{\wlip}[1]{{{\rm Lip}(#1)}}

\newcommand{\bsigma}{\boldsymbol{\sigma}}
\renewcommand{\bar}{\overline}
\newcommand{\size}{{\rm size}}
\newcommand{\sgn}{{\rm sgn}}

\title{Reducibility for a fast driven linear Klein-Gordon equation}

\author{
L. Franzoi\footnote{ International School for Advanced Studies (SISSA), Via Bonomea 265, 34136, Trieste, Italy \newline
 \textit{Email: } \texttt{luca.franzoi@sissa.it}} 
,
A. Maspero\footnote{ International School for Advanced Studies (SISSA), Via Bonomea 265, 34136, Trieste, Italy \newline
 \textit{Email: } \texttt{alberto.maspero@sissa.it}}
}

\begin{document}
	
\maketitle

\begin{abstract}
We prove a reducibility result for a linear Klein-Gordon equation with a quasi-periodic driving on a compact interval with Dirichlet boundary conditions.
   No assumptions are made on the size of the driving, however we require it to be fast oscillating. 
   In particular,  provided that the external frequency  is sufficiently large and chosen from a Cantor set of large measure, the original equation is conjugated to a time independent, diagonal one.
We achieve this  result in two steps. First, we perform a preliminary transformation, adapted to fast oscillating systems, which moves the original equation in a perturbative setting.
Then we show that this new equation can be put to constant coefficients by applying  a  KAM reducibility scheme, whose convergence requires a new type of   Melnikov conditions.
\end{abstract}

\section{Introduction}
We consider a linear Klein-Gordon  equation with quasi-periodic driving
\begin{equation}
	\label{eq:KG}
	\d_{tt}u-\d_{xx}u+\tm^2 u  + V(\omega t,x)u = 0 \ , \qquad x\in[0,\pi]\ , \ \ t\in\R \ ,
\end{equation}
with spatial Dirichlet boundary conditions $u(t, 0) = u(t, \pi) = 0$.\\
The potential $V:\T^\nu\times [0,\pi]\rightarrow \R$,  is quasi-periodic in time with a frequency vector $\omega\in\R^\nu\setminus\{0\}$.
The main feature of this driving  is that 
it is  not perturbative in size, but we require  it to be  
 fast oscillating, namely $\abs\omega\gg 1$.\\
The goal of our paper is to provide, for any frequency $\omega$ belonging to a Cantor set of large measure,  a reducibility result for the system \eqref{eq:KG}. That is, we  construct a change of coordinates which conjugates  equation  \eqref{eq:KG} into a diagonal, time independent   one.

As long as we know, this is the first result  of reducibility in an infinite dimensional setting in which the perturbation is not assumed  to be small in size, but only fast oscillating.

The proof is carried out in two steps, combining a preliminary transformation,  adapted to fast oscillating systems,  with a KAM reducibility scheme which completely removes the time dependence from the equation.
In particular we first perform a change of coordinates, following  \cite{abanin1}, that conjugates \eqref{eq:KG} to an equation with driving of size  $|\omega|^{-1}$, and thus perturbative in size. 
The price to pay is that the new equation 
might not fit in the standard KAM scheme developed by Kuksin in \cite{kuk87}.
 The problem is overcome in our model by exploiting the  pseudodifferential properties of the operators involved, showing that the new perturbation features regularizing properties. 
 
The second key ingredient of the proof concerns appropriate {\em balanced} Melnikov conditions (see \eqref{eq:meln_intro}), which allow us to perform a convergent KAM reducibility iteration. 

To carry out this program, we strongly exploit the fact  that the dispersion law of the system is asymptotically linear in the frequency space; 
this is used in a direct way to  prove the balanced Melnikov conditions, and in an indirect way to prove that the new forcing term generated by the preliminary transformation is a bounded operator (see also Remark \ref{rem:p}). 
This is the main reason why we consider the Klein-Gordon system.
  That being said, we suspect a similar result to be true also for systems with superlinear dispersion law, as the Schr\"odinger equation, but new ideas are needed to overcome the mentioned problems.

\vspace{1em}

From a mathematical point of view, our result is part of the attempts to extend classical Floquet theory  and  its quasi-periodic generalization to infinite dimensional systems. 
While many progresses have been made in the last 20 years to prove non perturbative reducibility for finite (and actually low) dimensional systems  
\cite{eli99,krik99, krik01, chav11, avila},
in the infinite dimensional case the only available results nowadays  deal with systems which are small perturbations of a diagonal operator, i.e. of the form  $D + \epsilon V(\omega t)$, where $D$ is diagonal, $\epsilon$ small and $\omega$ in some Cantor set.
In this case the literature splits essentially in two parts: the first one dealing with the case of perturbations which are  bounded operators \cite{  EK09,  GT11,   GP,   GP2, liangwang}, while the second one (of more recent interest) with   unbounded  ones \cite{ BG01, liuyuan, BBM14, FP15,  Bam16I, BGMR1}.\\
In particular, for the wave and Klein-Gordon  equations, 
the papers  \cite{ poschel, CY00, fang14, GP2} are in the  first group, while   \cite{ berti14, mont17} belong to the second one.
  In any case, all the previous results
    require a  smallness assumption on the size of the perturbation.

In order to deal with  perturbations that are periodic in time and fast oscillating, 
in \cite{ abanin2, abanin4, abanin1} Abanin, De Roeck, Ho and Huveneers developed an adapted normal form  that generalizes the classical Magnus expansion \cite{magnus}.
Such a normal form,  
which from now on  we call   {\em Magnus normal form}, allows to extract a time independent Hamiltonian (usually called the {\em effective} Hamiltonian), which approximates well the dynamics up to some finite but  very long times.
In \cite{abanin1}, the authors apply the Magnus normal form to the study of some quantum many-body systems (spin chains) with a fast periodic driving.
Although the Magnus normal form was developed for periodic systems, we extend it here for quasi-periodic ones and we use it as a preliminary transformation that moves the problem in a more favourable setting for starting a KAM reducibility scheme.
However, we point out that an important difference between \cite{abanin1} and our work lies in the fact that, while in \cite{abanin1} all the involved operators are bounded, on the contrary our principal operator is an unbounded one.

In case of systems of the form $H_0 + V(t)$, where the perturbation $V(t)$ is neither small in size nor fast oscillating, a general reducibility is not known. 
However, in same cases  it is possible to find some results of "almost reducibility"; that is, the original Hamiltonian is conjugated to one of the form $H_0 + Z(t) + R(t)$, where $Z(t)$ commutes with $H_0$, while $R(t)$ is an arbitrary smoothing operator, see e.g. \cite{BGMR2}.
This normal form ensures  upper bounds on the speed of transfer of energy from low to high frequencies; e.g. it implies that 
the  Sobolev norms of each  solution grows at most as $t^\epsilon$ when $t \to \infty$, for any arbitrary small $\epsilon >0$.
This procedure (or a close variant of it), has been applied also in   \cite{del2, MaRo,  Mon17}.  \\
There are also  examples in \cite{bou99, del, ma18} where the authors  
engineer periodic drivings  aimed to  transfer  energy from low to high frequencies and leading to  unbounded growth of Sobolev norms (see also Remark \ref{bou} below).

Finally, we want to mention also the papers 
\cite{BB08, CG17}, where KAM techniques are applied to construct quasi-periodic solutions with
$|\omega| \gg 1$. 
In   \cite{BB08} this is shown for 
 a nonlinear wave equation with Dirichlet boundary conditions, however  reducibility is not obtained. 
 In 
\cite{CG17}, KAM techniques are applied to a many-body  system with fast driving; 
the authors construct a periodic orbit with large frequency and prove its asymptotic stability.

\vspace{1em}
Before closing this introduction, we mention that  periodically driven systems have also a  great interest in physics, both theoretically and experimentally.
Indeed  such systems often exhibit a rich and  surprising behaviour, like the Kapitza pendulum \cite{kapi},  where the fast periodic driving stabilizes the otherwise unstable equilibrium point in which the pendulum is upside-down.
More recently, a lot of attention  has been dedicated to fast periodically driven  many-body systems \cite{eisert, goldman, kitagawa, jotzu}; here the interest is the possibility of  engineering periodic drivings for realizing novel quantum states of matter; this procedure, commonly called ``Floquet engineering'' \cite{bukov}, has been implemented in several physical systems, including   cold atoms,   graphenes and  crystals.

\subsection{Main result}
The potential driving $V(\omega t, x)$  is treated as a smooth function $V: \T^\nu \times [0, \pi]  \ni (\theta, x) \mapsto V(\theta, x)\in \R$,  $\nu \geq 1$,  which  satisfies  two conditions:
\begin{itemize}
	\item[(\b{V1})] The even extension in $x$ of $V(\theta,x)$ on the torus $\T\simeq [-\pi,\pi]$, which we still denote by $V$, is smooth in both  variables and it extends analytically in $\theta$ in a proper complex neighbourhood of $\T^\nu$ of width $\rho>0$. In particular, for any $\ell\in\N$, there is a constant $C_{\ell,\rho}>0$ such that
	\begin{equation*}
		\abs{\partial_x^\ell V(\theta,x)}\leq C_{\ell,\rho} \quad \forall\,x\in\T \ , \ \abs{{\rm Im}\,\theta}\leq \rho \ ;
	\end{equation*}
	\item[(\b{V2})] $\int_{\T^\nu}V(\theta,x)\wrt \theta = 0$   for any   $ x\in[0,\pi]$.
\end{itemize}

To state precisely our main result, equation \eqref{eq:KG} has to be rewritten as a Hamiltonian system. 
We introduce the new variables 
\begin{equation}
\label{}
\vf := B^{1/2}u  + \im B^{-1/2} \partial_t u \ , 
\qquad \bar \vf := B^{1/2}u  - \im B^{-1/2} \partial_t u \ ,
\end{equation}
where
\begin{equation}
\label{def:B}
B:=\sqrt{-\Delta +\tm^2} \ ;
\end{equation}
note that   the operator $B$ is invertible also when $\tm  = 0$, since we consider Dirichlet boundary conditions. In the new variables
 equation \eqref{eq:KG} is equivalent to
\begin{equation}\label{eq:KG_c}
\im\partial_t \vf (t) = B\vf(t)+\frac{1}{2}\,B^{-1/2}V(\omega t)B^{-1/2}(\vf(t) +\overline{\vf}(t)) \  . 
\end{equation}
Taking \eqref{eq:KG_c} coupled with its complex conjugate, we obtain the following system
\begin{equation}\label{eq:KG_matrix}
	\im \partial_t \varphi(t) = \bH(t)\varphi(t) \ , \quad \bH(t):=\left(\begin{matrix}
	B & 0 \\ 0 & -B
	\end{matrix}\right) + \frac{1}{2}\,B^{-1/2}V(\omega t, x )B^{-1/2}\left(\begin{matrix}
	1 & 1 \\ -1 & -1
	\end{matrix}\right) \ , 
\end{equation}
where, abusing notation, we denoted $ \varphi(t)\equiv \left(\begin{matrix}
	\varphi(t) \\ \bar{\varphi}(t)
	\end{matrix}\right)$ the vector with the components  $\vf, \bar \vf$.
The phase space  for \eqref{eq:KG_matrix}  is $\cH^r\times\cH^r , $
 where, for $r\geq 0$, 
\begin{equation}\label{eq:sine_sob}
	 \cH^r:=\Set{ \vf(x)=\sum_{m\in\N} \vf_m\sin(mx), \ x\in [0,\pi] | \norm{\vf}_{\cH^r}^2:=\sum_{m\in\N}\braket{m}^{2r}\abs{\vf_m}^2<\infty  }  \ .
\end{equation}
Here we have used the notation $\braket{m}:=(1+\abs m^2)^{\frac{1}{2}}$, which will be kept throughout all the article.	
We define the $\nu$-dimensional annulus of size $\tM>0$ by
$$R_{\tM}:=\bar{B_{2\tM}(0)}\backslash B_{\tM}(0)\subset \R^{\nu} \ ;$$
here we denoted by $B_M(0)$ the ball of center zero and radius $M$ in the Euclidean topology of $\R^\nu$.
\begin{thm}
\label{thm:main}
	Consider the system \eqref{eq:KG_matrix} and assume (\b{V1}) and (\b{V2}).
	 Fix arbitrary $r, \tm \geq 0$ and $\alpha\in(0,1)$. Fix also an arbitrary $\gamma_*>0$ sufficiently small.	 \\
	Then there exist $\tM_* >1 $, $C >0$ and,
 for any $\tM\geq \tM_*$, a subset $\Omega_{\infty}^\alpha=\Omega_\infty^\alpha(\tM,\gamma_*)$ in $R_\tM$, fulfilling
	\begin{equation}
		\frac{\meas(R_{\tM}\backslash\Omega_\infty^\alpha)}{\meas(R_{\tM})} \leq  C \gamma_* , 
	\end{equation}
	such that the following holds true.	 
	For any frequency vector $\omega\in\Omega_\infty^\alpha$, there exists an  operator $\cT(\omega t; \omega)$, bounded in $\cL(\cH^r\times \cH^r)$, quasi-periodic in time and analytic  in a shrunk neighbourhood of $\T^\nu$ of width $\rho/8$, such that the change of coordinates $\vf = \cT(\omega t; \omega)\psi$  conjugates \eqref{eq:KG_matrix} to the diagonal time-independent system
	\begin{equation}\label{eq:eff_sys}
		\im \dot{\psi}(t)=\bH^{\infty,\alpha}\psi(t) \ , \quad \bH^{\infty,\alpha}:=\left( \begin{matrix}
		D^{\infty,\alpha} & 0 \\ 0 & -D^{\infty,\alpha}
		\end{matrix} \right) \ , \ D^{\infty,\alpha}=\diag\Set{\lambda_j^{\infty}(\omega) | j \in\N} .
	\end{equation}
	The transformation $\cT(\omega t;\omega)$ is close to the identity, in the sense that there exists $C_r>0$ independent of $\tM$ such that
	\begin{equation}
		\norm{\cT(\omega t;\omega)-\uno}_{\cL(\cH^r\times\cH^r)} \leq \frac{C_r}{ \tM^{\frac{1-\alpha}{2}}} \ .
	\end{equation}
	The new  eigenvalues $(\lambda_j^{\infty}(\omega))_{j\in\N}$ are real, Lipschitz in $\omega$,  and admit the following asymptotics for $j\in\N$:
	\begin{equation}\label{eq:eff_eig}
		\lambda_j^{\infty}(\omega) \equiv \lambda_j^{\infty}(\omega,\alpha) = \lambda_j + \varepsilon_j^{\infty}(\omega,\alpha) \ , \quad \ \varepsilon_j^{\infty}(\omega,\alpha)\sim O\left( \frac{1}{\tM j^{\alpha}} \right)  \ ,
	\end{equation}
	where $\lambda_j = \sqrt{j^2 + \tm^2}$ are the eigenvalues of the operator $B$.
\end{thm}
\begin{rem}
	In particular, back to the original coordinates, equation \eqref{eq:KG} is reduced to
	\begin{equation}
		\d_{tt}u+\left(D^{\infty,\alpha}\right)^2u = 0 \ .
	\end{equation}
	\end{rem}
\begin{rem}
	The parameter $\alpha$, which one chooses and fixes in the real interval $(0,1)$, influences the asymptotic expansion of the final  eigenvalues,  as one can read from \eqref{eq:eff_eig}. Also the construction of the set of the admissible frequency vectors heavily depends on this parameter.
\end{rem}
\begin{rem}
We believe that  the assumptions of Theorem \ref{thm:main} can be weakened, for example asking only Sobolev regularity for $V(\theta, x)$, dropping ({\bf V2}) or using periodic boundary conditions; these issues  will be addressed  elsewhere.
\end{rem}
\begin{rem}
In Theorem \ref{thm:main} we can take also $\tm = 0$; this is due to the fact that, with Dirichlet boundary conditions, the unperturbed eigenvalues $\lambda_j$ are simple,  integers and  their corrections are small (see \eqref{eq:eff_eig}). 
This implies that it is enough to  move the frequency vector $\omega$ for avoiding resonances.
\end{rem}

Let  us denote by $\cU_{ \omega}( t,\tau)$ the propagator generated by \eqref{eq:KG_matrix}  such that $\cU_{ \omega}(\tau, \tau)=\uno$, $\forall \tau \in \R$.
An immediate consequence of Theorem \ref{thm:main} is that we have a Floquet   decomposition:
\begin{equation}\label{decom}
 \cU_{\omega}( t,\tau) =  \cT(\omega t; \omega)^* \circ  {\rm e}^{-\im(t-\tau)\bH^{\infty, \alpha}} \circ  \cT(\omega\tau; \omega) \ .
\end{equation}

Another consequence of \eqref{decom} is  that, for any $r \geq0$, the norm $\norm{\cU_{\omega}(t,0) \vf_0}_{\cH^r \times \cH^r}$  is bounded uniformly in time:
\begin{cor}
\label{cor.1}
Let $\tM \geq \tM_*$ and $\omega \in \Omega_\infty^\alpha$. 
For any $r \geq 0$ one has
\begin{equation}\label{unest}
 c_{r}\norm{\vf_0}_{\cH^r\times \cH^r} \leq\norm{ \cU_{\omega}( t,0)\vf_0}_{\cH^r\times \cH^r}
  \leq C_{r}\norm{\vf_0}_{\cH^r\times \cH^r} ,\quad \forall t\in\R \ , 
 \forall\vf_0\in {\cH^r\times \cH^r}, 
\end{equation}
for some  $ c_{r}>0, C_{r}>0$.\\
More precisely,  there exists a constant $c'_{r}>0$ s.t. if the initial data $\vf_0 \in \cH^{r}\times \cH^r$ then
\begin{equation*}
\left(1- \frac{c'_r}{ \tM^{\frac{1-\alpha}{2}}}\right)\norm{\vf_0}_{\cH^r\times \cH^r}   \leq \norm{\cU_{\omega}( t,0)\vf_0}_{\cH^r\times \cH^r} \leq 
\left(1+ \frac{c'_r}{ \tM^{\frac{1-\alpha}{2}}} \right)
\norm{\vf_0}_{\cH^r\times \cH^r} , \quad \forall t\in\R \ .
\end{equation*}
\end{cor}

\begin{rem}
\label{bou}
Corollary \ref{cor.1} shows that, if the frequency $\omega$ is chosen in the Cantor set $\Omega^\alpha_\infty$,  no phenomenon of  growth of Sobolev norms can happen. 
On the contrary, if $\omega$ is chosen  resonant, one can construct drivings which provoke  norm explosion with exponential rate, see  \cite{bou99} (see also \cite{ma18} for other examples).
\end{rem}

\begin{rem}
For nonlinear PDEs, the property that all solutions have uniformly  bounded Sobolev norms is  typical connected to integrability.
For example, the 1 dimensional defocusing NLS, the KdV and Toda chain exhibit this property (see e.g. \cite{AlbertoVeyPaper, BambusiM16, Kappeler16}).
\end{rem}

\subsection{Scheme of the proof}
Our proof  splits into three different parts, which we now summarize.

\paragraph*{The Magnus normal form. } In Section \ref{sec:magnus} we perform a preliminary transformation, adapted to fast oscillating systems, which moves the non-perturbative equation \eqref{eq:KG_matrix} into a pertubative one where the size of the transformed quasi-periodic potential is as small as large is the module of the frequency vector. Sketchily, we perform a change of coordinates  which conjugates 
\begin{equation}
\label{sss}
	\left\{ \begin{matrix}
	\bH(t)=\bH_0+\bW(\omega t) \\
	"\size(\bW) \sim 1"
	\end{matrix} \right. \quad  \rightsquigarrow \quad  \left\{ \begin{matrix}
	\wt\bH(t)=\bH_0+\bV(\omega t;\omega) \\
	"\size(\bV) \sim |\omega|^{-1}"
	\end{matrix} \right. \ .
\end{equation}
This change of coordinates, called below Magnus normal form, is an extension to quasi-periodic systems of the one performed in  \cite{abanin1}.  Note that  $\bH_0$ is the same on both sides of \eqref{sss} provided $\int_{\T^\nu}\bW(\theta) \di \theta = 0$, which is fulfilled in our case thanks to Assumption {\bf (V2)}.

As we already mentioned, the price to pay  is that, in principle, it is not clear that the new perturbation is  sufficiently regularizing to 
fit in a standard KAM scheme  (see Remark \ref{rem:p} for a more detailed discussion).
 
Here it is essential to employ  pseudodifferential calculus, thanks to which we control the order (as a pseudodifferential operator) of the new perturbation, 
and prove that it is actually enough regular for the KAM iteration. 
This is true because the principal term of the new perturbation is a commutator with $\bH_0$   (see equation \eqref{eq:magnus_4}), and 
one can exploit the smoothing properties of the commutator of  pseudodifferential operators.

\paragraph*{Balanced Melnikov conditions.} 
After the Magnus normal form, we   perform a  KAM reducibility scheme in order to remove the time dependence on the coefficients of the equation.  As usual one needs  second order Melnikov conditions  on the unperturbed eigenvalues $\lambda_j = \sqrt{j^2 + \tm^2}$. 
One might impose that for some $\gamma, \tau >0$, 
\begin{equation}\label{eq:meln_0}
	\abs{\omega\cdot k + \lambda_j - \lambda_l} \geq \frac{\gamma}{\braket{k}^\tau}\frac{\braket{j- l}}{|\omega|} \ , \quad \forall (k, j, l) \in  \Z^\nu\times \N \times \N , \ 
\ 	(k,j,l) \neq (0,j, j)   \ ;
\end{equation}
such conditions are violated  for a set of  frequencies  of relative measure bounded by $C  \gamma$, where $C$ is a constant independent of $|\omega|$\footnote{remark that the conditions $\abs{\omega\cdot k + \lambda_j \pm \lambda_l} \geq \frac{\gamma}{\braket{k}^\tau}{\braket{j\pm l}}$ are violated  on a set of relative measure $\sim \gamma |\omega|$, which is as large as the size of the   frequency vector.}.\\
These   Melnikov conditions are  useless in our context; indeed recall that, after the Magnus normal form, the new perturbation has size $\sim |\omega|^{-1}$ while  the small denominators in \eqref{eq:meln_0} have size  $\sim |\omega|$; so  the two of them compensate each others, and the KAM step cannot reduce in size.\\
To overcome the problem, rather than  \eqref{eq:meln_0}, we  impose new  {\em balanced} Melnikov conditions, in which   we   balance the loss in size (in the denominator) and gain in regularity (in the numerator) in  \eqref{eq:meln_0}.
More precisely, we show that  for any $\alpha \in [0,1]$ one can impose 
\begin{equation}\label{eq:meln_intro}
	\abs{\omega\cdot k + \lambda_j - \lambda_l} \geq \frac{\gamma}{\braket{k}^\tau}\frac{\braket{j - l}^\alpha}{|\omega|^\alpha} \ , \quad \forall (k, j, l) \in  \Z^\nu\times \N \times \N , \ 
\ 	(k,j,l) \neq (0,j, j) 
\end{equation}
for a set of $\omega$'s in $R_\tM$ of large relative measure. This is proved in Section \ref{sub:um}.  By choosing $0 <\alpha < 1$,  the l.h.s. of \eqref{eq:meln_intro} is larger than the corresponding one in  \eqref{eq:meln_0}, and the KAM transformation reduces in size.
However note that the choice of $	\alpha$ will influence the regularizing effect given by $\la j\pm l \ra^\alpha$ in the r.h.s. of \eqref{eq:meln_intro}; ultimately, this  modifies the asymptotic expansion of the final eigenvalues, as one can see in 
\eqref{eq:eff_eig}.

\paragraph*{The KAM reducibility.} At this point we perform a  KAM reducibility scheme; this step is nowadays quite standard and we only sketch the proofs.


\vspace{2em}
\noindent{\bf Acknowledgments.} We thanks Dario Bambusi, Massimiliano Berti, Roberto Feola,  Matteo Gallone and Vieri Mastropietro for many stimulating discussions.
 We were partially supported by  Prin-2015KB9WPT and  Progetto  GNAMPA - INdAM 2018 ``Moti stabili ed instabili in equazioni di tipo Schr\"odinger''.

\section{Functional settings}\label{sec:fun}
Given  a set  $\Omega \subset \R^\nu$ and a Fréchet space $\cF$, the latter endowed with a system of seminorms $\{\norm \cdot _n  | \ n\in\N\}$, we define for a function $f: \Omega \ni \omega \mapsto f(\omega) \in \cF$ 
 the quantities 
\begin{equation}
\label{lip}
\abs{f}_{n, \Omega}^\infty:=
\sup_{\omega \in \Omega} \norm{f(\omega)}_n
\ , \qquad
\abs{f}_{n,\Omega}^{\lip} :=
\sup_{\omega_1, \omega_2 \in \Omega \atop \omega_1 \neq \omega_2 }
\frac{\norm{f(\omega_1)- f(\omega_2)}_n}{\abs{\omega_1 - \omega_2}} .
\end{equation}
Given 
 $\tw \in \R_+$,  we 
 denote by $\lip_\tw(\Omega, \cF)$ the space of functions from $\Omega$ into $\cF$  such that 
 \begin{equation}\label{eq:wlip_def}
	\norm f_{n,\Omega}^{\wlip{\tw}} := \abs f_{n,\Omega}^\infty + \tw \abs f _{n,\Omega}^{\lip} < \infty \ .
\end{equation}


\subsection{Pseudodifferential operators}\label{sub:pseudo}
The main tool for the construction of the Magnus transform in Section \ref{sec:magnus} is the calculus with pseudodifferential operators acting on the scale of the standard Sobolev spaces on the torus $\T:=\R / 2\pi\Z$, which is defined for any $r \in \R$ as 
\begin{equation}
	H^r(\T):=\Set{ \vf(x) = \sum_{j\in\Z}\vf_j e^{\im jx}, \ x\in \T |  \norm{\vf}_{H^r(\T)}^2:= \sum_{j\in\Z}\braket{j}^{2r}\abs{\vf_j}^2 <\infty  } .
\end{equation} 
For a function $f:\T\times \Z \to \R$, define the difference operator $\triangle f(x,j) := f(x,j+1)-f(x,j)$ and let $\Delta^\beta=\Delta\circ...\circ\Delta$ be the composition  $\beta$ times of $\Delta$. 
Then, we have the following:
\begin{defn}\label{defn:symbols_semi}
	We say that a function $f:  \T \times \Z \to \R$ is a symbol of order $m \in \R$ if for any $j \in \Z$ the map  $ x\mapsto f( x, j)$ is smooth and, furthermore, for any
	$ \alpha, \beta \in \N$, there exists $C_{\alpha, \beta}>0$ such that
	$$
	\abs{\partial_x^\alpha \triangle^\beta f( x, j)} \leq C_{\alpha, \beta} \, \la j \ra^{m - \beta}  \ , \quad \forall x \in \T   \ .
	$$
	If this is the case, we write $f \in S^m$. 
\end{defn}
We endow $S^m$  with the family of seminorms
$$
\wp^{m}_\ell(f) :=   
 \sum_{\alpha + \beta \leq \ell}
\\ \sup_{(x, j) \in \T \times  \Z} \la j \ra^{-m + \beta}
\left|\partial_x^\alpha \, \triangle^\beta f(x,j)\right| \ , \quad \ell \in \N_0 \ .
$$
\paragraph*{Analytic families of pseudodifferential operators.}
We will consider in our discussion also symbols depending real analytically on the variable $\theta \in \T^\nu$. To define them, we need to introduce the complex neighbourhood of the torus
$$
\T^\nu_\rho := \Set{a+\im b \in \C^\nu | a\in\T^\nu \, , \ \abs b \leq \rho  } \ .
$$
\begin{defn}\label{defn:symbols_semi2}
	Given $m \in \R$ and $\rho >0$,  a function $f: \T^\nu \times  \T \times \Z \to \R$, $(\theta, x, j) \mapsto f(\theta, x,j)$, is called a symbol of class $S^m_\rho$ if for any $j \in \N$ it is smooth in $x$, it  
	extends analytically in $\theta$ in $\T^\nu_\rho$
	and, furthermore, for every $\alpha, \beta \in \N$ there exists $C_{\alpha, \beta}>0$ such that 
	$$
	\abs{\partial_x^\alpha \triangle^\beta f(\theta, x, j)} \leq C_{\alpha, \beta} \, \la j \ra^{m - \beta}   \ \quad \forall x \in \T \ , \ \forall\,  \theta \in \C^\nu, \ \abs{{\rm Im} \,  \theta} \leq \rho \ .
	$$
	For such a function we write $f \in S^m_\rho$. 
\end{defn}
We endow the class $S^m_\rho$  with the family of seminorms
$$
\wp^{m, \rho}_\ell(f) := 
 \sup_{|{\rm Im } \, \theta| \leq \rho} \ \ 
 \sum_{\alpha + \beta \leq \ell}
\\ \sup_{(x, j) \in \T \times  \Z} \la j \ra^{-m + \beta}
\left|\partial_x^\alpha \, \triangle^\beta f(\theta, x,j)\right| \ , \quad \ell \in \N_0 \ .
$$
We associate to a symbol $f \in S^m_\rho$ the operator $f(\theta, x, D_x)$ by standard quantization
\begin{equation}
	\psi(x)=\sum_{j\in\Z}\psi_j e^{\im jx} \ \mapsto \ \left(f(\theta,x,D_x)\psi \right)(x):= \sum_{j\in\Z}f(\theta,x,j)\psi_j e^{\im jx} \ ;
\end{equation}
here $D_x=D:=\im^{-1}\partial_x$ is the Hörmander derivative.
\begin{defn}\label{def:pseudo}
We say that  $F\in \cA^m_\rho$ if it is a pseudodifferential operator
with symbol of class $S^m_\rho$, i.e. if there
exists a symbol $f \in S^m_\rho$ such that $F = f(\theta, x, D_x)$.\\
If $F$ does not depend on $\theta$, we simply write $F\in\cA^m$.
\end{defn}
\begin{rem}
For any $\sigma \in \R$, the operator $\la D \ra^\sigma \equiv  \left(1 - \partial_{xx}\right)^\frac{\sigma}{2}$ is  in $\cA^\sigma$.
\end{rem}
As usual we give to  $\cA^m_\rho$ a Fr\'echet structure by endowing it with the seminorms of the symbols. 
Finally we define the class of pseudodifferential operators depending on a Lipschitz way on an external parameter.
\begin{defn}\label{def:pseudo_lip}
	We denote by $\lip_\tw(\Omega, \cA^m_\rho)$ the space of pseudodifferential operators whose symbols belong to $\lip_\tw(\Omega, S^m_\rho)$ and by $ \left( \wp_j^{n,\rho}(\cdot)_\Omega^\wlip{\tw} \right)_{j\in\N} $ the corresponding seminorms.
\end{defn}
\begin{rem}
\label{rem:comp}
Let $F \in \lip_\tw(\Omega,\cA_\rho^m)$ and $G \in \lip_\tw(\Omega,\cA_\rho^n)$. Then the symbolic calculus implies that
$FG \in  \lip_\tw(\Omega,\cA_\rho^{m+n})$ and 
$[F,G] \in  \lip_\tw(\Omega,\cA_\rho^{m+n-1})$, with the quantitative bounds
\begin{align*}
\forall j \ \ \exists N \ \text{s.t. } \  &  \wp_j^{m+n, \rho}(FG)_\Omega^\wlip{\tw} \leq C_1
\wp_N^{m, \rho}(F)_\Omega^\wlip{\tw} \, \wp_N^{n,\rho}(G)_\Omega^\wlip{\tw} \ , \\
\forall j \ \ \exists N \ \text{s.t. } \  &  \wp_j^{m+n-1, \rho}([F,G])_\Omega^\wlip{\tw} \leq C_2 
\wp_N^{m,\rho}(F)_\Omega^\wlip{\tw} \, \wp_N^{n,\rho}(G)_\Omega^\wlip{\tw} \ . 
\end{align*}  

\end{rem}

\paragraph*{Parity preserving operators.} The space $\cH^0$ of \eqref{eq:sine_sob} is naturally identified with the subspace of $H^0(\T) \equiv L^2(\T)$ of odd functions. Therefore it makes sense to work with pseudodifferential operators preserving the parity. Before describing them, we recall the orthogonal decomposition of the periodic $L^2$-functions on $\T$:
$$ L^2(\T) = L_{even}^2(\T)\oplus L_{odd}^2(\T)  $$ where, for $u(x)=\sum_{j\in\Z}u_je^{\im jx} \in L^2(\T)$, we have for any $j\in\Z$,
\begin{equation}
	u\in L_{even}^2(\T) \ \Leftrightarrow \ u_{-j}=u_j \quad \text{ and } \quad u\in L_{odd}^2(\T) \ \Leftrightarrow \ u_{-j}=-u_j . 
\end{equation}
\begin{defn}
	We denote by $\cP S_\rho^m$ the class of symbols $f\in S_\rho^m$ satisfying the property
	\begin{equation}\label{eq:pp_final}
	f(\theta,x,j)=f(\theta,-x,-j) \quad \ \forall \theta \in\T^\nu \, , \ x\in\T \, , \ j\in\Z \ .
	\end{equation}
	We denote by   $\cP\cA_\rho^m$ the subset of $\cA^m_\rho$ of parity preserving operators, that is, those operators $A\in\cA_\rho^m$ such that $A(L_{even}^2)\subseteq L_{even}^2$ and $A(L_{odd}^2)\subseteq L_{odd}^2$.
\end{defn}
\begin{lem}\label{lem:pp_class}
	Let $F\in\cA_\rho^m$ with symbol $f\in S_\rho^m$. Then $F\in\cP\cA_\rho^m$ if and only if $f\in\cP S_\rho^m$.
\end{lem}
\begin{proof}
It is easy to check that $F(L^2_{odd}(\T)) \subseteq L^2_{odd}(\T)$ if and only if the symbol  $f(x, j)$ of $F$ fulfills 
${\rm Im}[(f(x,j)-f(-x,-j))e^{\im jx} ] \equiv  0 . $
Similarly $F(L^2_{even}(\T)) \subseteq L^2_{even}(\T)$ 
 if and only if  
 $ {\rm Re} [(f(x,j)-f(-x,-j))e^{\im jx}] \equiv  0 .
 $
 \end{proof}

\begin{rem}
\label{DPA}
For all $\sigma \in \R$, the operator $\la D \ra^{\sigma} \in \cP\cA^\sigma$, while, by the assumption (\b{V1}), $V \in \cP\cA_\rho^0$.
\end{rem}
\begin{rem}\label{rem:comp_pp}
Parity preserving operators are  closed under composition and commutators.
\end{rem}
\begin{rem}
\label{rem:B-1}
For $\tm = 0$ and $\sigma >0$, we define 
$B^{-\sigma}\psi := \sum_{j \neq 0} \frac{1}{|j|^\sigma} \psi_j e^{\im j x}$ for any $\psi \in L^2(\T)$; clearly $B^{-\sigma} \in \cP \cA^{-\sigma}$. 
Note that 
$B  B^{-1} \psi = B^{-1} B \psi = \psi - \psi_0$. However, the restriction $B\vert_{\cH^0}$ of $B$ to the  phase space \eqref{eq:sine_sob} is invertible (since the phase space contains only functions with zero average) and $B^{-1}$ is its inverse. 
\end{rem}

\subsection{Matrix representation and operator matrices}
\label{sec:mo}
For the KAM reducibility, a second and wider class of operators without a pseudodifferential structure is needed on the scale of Hilbert spaces $\left( \cH^r \right)_{r\in\R}$, as defined as in \eqref{eq:sine_sob}. Moreover, let $\cH^\infty:=\cap_{r\in\R}\cH^r$ and $\cH^{-\infty}:=\cup_{r\in\R}\cH^r$.
If $A$ is a linear operator, we denote by $A^*$ the adjoint of $A$ with respect to the scalar product of $\cH^0$,
while we denote by $\bar A$ the conjugate operator:
$\bar A \psi := \bar{A\bar \psi}$ $\, \forall \psi \in D(A).$
\paragraph*{Matrix representation of operators.}
To any linear operator $A\colon \cH^\infty \to \cH^{-\infty}$  we  associate its matrix of coefficients $(A_m^n)_{m,n \in \N} $ on the basis $(\wh\be_n:=\sin(nx))_{n \in \N}$,  defined for $m,n\in\N$ as
$$
 {A}_m^n \equiv  \la A \wh\be_m , \wh\be_n \ra_{\cH^0}  \ .
$$
\begin{rem}
If $A$ is a bounded operator, the following implications hold:
\begin{align*}
& A  = A^* \Longleftrightarrow  A_m^n  = \bar{ A_n^m}
\ \ \  \forall m,n \in \N \ ; \\
& \bar A  = A^* \Longleftrightarrow  A_m^n =  A_n^m 
\  \ \ \forall m,n \in \N  \ .
\end{align*}
\end{rem}
A useful norm we can put on the space of such operators is in the following:
\begin{defn}\label{def:sdecay}
Given a linear operator $A\colon \cH^\infty \to \cH^{-\infty}$ and  $s \in \R$,  we say that $A$ has finite $s$-decay norm provided
\begin{equation}\label{eq:sdecay_norm}
\abs{A}_s:=\left( \sum_{h\in \N_0 }\la h \ra^{2s}\sup_{|m-n|=h}\abs{A_m^n}^2 \right)^{1/2} < \infty \ .
\end{equation}
\end{defn}
One has the following:
\begin{lem}[Algebra of the s-decay]\label{rem:alg}
For any $s > \frac{1}{2}$ there is a constant $C_s>0$ such that
\begin{equation}\label{eq:algebra_sdecay}
	\abs{AB}_s\leq C_s \abs A_s\abs B_s  .
\end{equation} 
\end{lem}
\noindent 
The proof of the Lemma is an easy variant of the one in 
\cite{bebo13}
 we sketch it in  Appendix \ref{alg.s.decay}.
\begin{rem}\label{rem:opnorm_vs_sdecay}
If $A: \cH^\infty \to \cH^{-\infty}$ has finite $s$-decay norm with $s > \frac{1}{2}$, then for any $r\in[0,s]$, 
$A$ extends to a bounded operator $\cH^{r} \to \cH^{r}$.  Moreover, by tame estimates, one has the quantitative bound
$ \norm{A}_{\cL(\cH^r)} \leq C_{r,s} |A|_s$. 
\end{rem}

Next, we consider operators depending analytically on angles $\theta \in \T^\nu$.
\begin{defn}
\label{def:mat}
Let $A$ be a $\theta$-depending  operator, $A \colon \T^\nu \to \cL(\cH^\infty, \cH^{-\infty})$. 
Given  $ s \geq 0$ and $\rho>0$, we say that 
$A \in \cM_{ \rho, s}$ if  one has 
	\begin{equation}
	\label{norm:rs}
		\abs A_{\rho,s}:=\sum_{k \in \Z^{\nu}}e^{\rho\abs k}\abs{\wh A(k)}_s < \infty \ , \quad \mbox{where } \quad \wh A(k) := \frac{1}{(2\pi)^\nu} \int_{\T^\nu} A(\theta) \, e^{- \im k \cdot \theta } \, \di \theta \ .
	\end{equation}
\end{defn}

\begin{rem} 
\label{rem:coef.mat}
If $A$ is a $\theta$-depending bounded operator, the following implications hold:
\begin{align*}
& A  = A^* \Longleftrightarrow [\wh A(k)]^* = \wh A(-k)  \ \forall k \in \Z^\nu 
\Longleftrightarrow \wh A_m^n(k) = \bar{\wh A_n^m(-k)}
\ \forall k \in \Z^\nu , \ \forall m,n \in \N \\
& \bar A  = A^* \Longleftrightarrow [\wh A(k)]^* = \bar{\wh A(k)}  \ \forall k \in \Z^\nu 
\Longleftrightarrow \wh A_m^n(k) = {\wh A_n^m(k)}
\ \forall k \in \Z^\nu , \ \forall m,n \in \N
\end{align*}
\end{rem}
If $\Omega \ni \omega \mapsto A(\omega)\in\cM_{\rho, s}$ is a Lipschitz map,  we write $A \in \lip_\tw(\Omega, \cM_{\rho,s})$, provided 
\begin{equation}
\label{lip.msr}
\abs{A}_{\rho, s,\Omega}^{\wlip{\tw}}  :=
\sup_{\omega \in \Omega} \abs{A(\omega)}_{\rho, s}
+\tw \sup_{\omega_1 \neq \omega_2 \in \Omega} \frac{\abs{A(\omega_1) - A(\omega_2)}_{\rho, s}}{\abs{\omega_1 - \omega_2}}	 < \infty \ .
\end{equation}
\begin{rem}
\label{rem:ana}
For any $s > \frac{1}{2}$ and $\rho>0$, the spaces $\cM_{\rho, s}$ and
$\lip_\tw(\Omega, \cM_{\rho,s})$ 
 are closed with respect to composition, with  
$$
\abs{AB}_{\rho, s} \leq C_s \abs{A}_{\rho, s} \, \abs{B}_{\rho, s} , 
\qquad
\abs{AB}_{\rho, s, \Omega}^{\wlip{\tw}} \leq C_{s} \abs{A}_{\rho, s, \Omega}^{\wlip{\tw}} \, \abs{B}_{\rho, s, \Omega}^{\wlip{\tw}}  .
$$
This  follows from Lemma \ref{rem:alg}  and the algebra properties for analytic functions.
\end{rem}

\paragraph{Operator matrices.} We are going to meet matrices of operators of the form 
\begin{equation}
\label{eq:Amatrix}
\bA=\left( \begin{matrix}
A^{d} & A^o \\ -\bar{A^o} & -\bar{A^d} 
\end{matrix} \right) \ , 
\end{equation}
where $A^d$ and $A^o$ are linear  operators belonging to the class $\cM_{\rho, s}$.
Actually, the operator $A^d$ on the diagonal will have different decay properties than the element on the anti-diagonal $A^o$. Therefore, we introduce 
classes of operator matrices in which we keep track of these differences. 
\begin{defn}
	Given an operator matrix $\bA$ of the form
	\eqref{eq:Amatrix}, $\alpha, \beta \in \R$, $\rho>0$,$s \geq 0$, we say that 
	$A $ belongs to $\cM_{\rho,s}(\alpha, \beta)$ if
	\begin{equation}\label{struttura}
		[A^d]^* = A^d \ , \qquad [A^o]^* = \bar{A^o}
	\end{equation}	
	and one also has 
	\begin{align}\label{M1}
		&	\la D \ra^\alpha \,  A^d  \ , \ A^d \, \la D \ra^{\alpha}  \in \cM_{\rho,s}  \ , \\
		\label{M2}
		&	\la D \ra^\beta \,  A^o  \ , \ A^o \, \la D \ra^{\beta}  \in \cM_{\rho, s}  \ , \\
		\label{M3}
		&	\la D \ra^\sigma \,  A^\delta \, \la D \ra^{- \sigma}  \in \cM_{\rho, s} \ , \quad \forall \sigma \in \{ \pm \alpha, \pm \beta, 0 \} \ , \ \ \forall \delta \in \{d, o\} \ .
	\end{align}
	We endow $\cM_{\rho,s}(\alpha, \beta)$  with the  norm 
	\begin{equation}\label{eq:sdecay_matrix_norm}	
		\begin{aligned}
		\abs{\bA}_{\rho, s}^{\alpha, \beta} := & 
		\abs{\braket{D}^{\alpha}A^d}_{\rho,s}
		+\abs{A^d\braket{D}^{\alpha}}_{\rho, s} 
		+\abs{\braket{D}^{\beta}A^o}_{\rho,s}
		+\abs{A^o\braket{D}^{\beta}}_{\rho, s} \\	
		& 
		+ \sum_{\sigma\in\{ \pm \alpha, \pm \beta,  0 \} \atop \delta \in \{d, o\}} \abs{\braket{D}^{\sigma}A^\delta \braket{D}^{-\sigma}}_{\rho,s} \ ,
		\end{aligned}
	\end{equation}
	with the convention that, in case of repetition (when $\alpha=\beta$, $\alpha=0$ or $\beta=0$), the same terms are not summed twice.
	When $\bA$ is independent of $\theta\in\T^\nu$, we use the norm  $\abs\bA_s^{\alpha,\beta}$, defined as \eqref{eq:sdecay_matrix_norm}, but replacing $\abs{\cdot }_{\rho, s}$ with the $s$-decay norm $\abs{\cdot}_s$ defined in \eqref{eq:sdecay_norm}.
\end{defn}
Let us motivate the properties describing the class $\cM_{\rho, s}(\alpha, \beta)$:
\begin{itemize}
	\item Condition \eqref{struttura} is equivalent to ask that $\bA$ is the Hamiltonian vector field of a real valued quadratic Hamiltonian, see e.g. \cite{mont17} for a discussion;
	\item Conditions \eqref{M1} and \eqref{M2} control the decay properties for the coefficient of the coefficients of the matrices associated to $A^d$ and $A^o$: indeed the matrix coefficients of $\la D \ra^{\alpha} A \,  \la D\ra^{\beta} $ are given by
	$$
	\left[ \wh{\la D \ra^{\alpha} A \,  \la D\ra^{\beta}} \right]_{m}^n(k)   = \la m \ra^\alpha \, \wh{A}_m^n(k) \, \la n \ra^\beta \ , 
	$$
	therefore decay (or growth) properties for the matrix coefficients of the operator $A$ are implied by the boundedness of the norms $| \cdot |_{\rho,s}$;
	\item Condition \eqref{M3} is just for  simplifying some   computations below.
\end{itemize}

%
%
%
%
%
\begin{rem}
	Let $0<\rho'\leq \rho$, $0 \leq s' \leq s $ $\alpha\geq \alpha'$, $\beta\geq \beta'$. 
	Then $\cM_{\rho,s}(\alpha, \beta)\subseteq \cM_{\rho',s'}(\alpha', \beta')$ with the quantitative bound
	$\abs{\bA}_{\rho', s'}^{\alpha', \beta'} \leq \abs{\bA}_{\rho, s}^{\alpha, \beta}$.
\end{rem}

Finally, if	$A^d(\omega)$ and $A^o(\omega)$ depend in a Lipschitz way on a parameter $\omega$, we introduce the Lipschitz norm
\begin{equation}\label{eq:lip_ab_def}
	\abs{\bA}_{\rho,s, \alpha, \beta,\Omega}^{\wlip{\tw}}  :=
	\sup_{\omega \in \Omega} \abs{\bA(\omega)}_{\rho, s}^{\alpha, \beta}
	+\tw \sup_{\omega_1 \neq \omega_2 \in\Omega}
	 \frac{\abs{\bA(\omega_1)- \bA(\omega_2)}_{\rho, s}^{\alpha,\beta}}{\abs{\omega_1 - \omega_2}}	\ .
\end{equation}
If such a norm is  finite, we write $\bA \in \lip_\tw(\Omega, \cM_{\rho,s}(\alpha, \beta))$.
\paragraph*{Embedding of parity preserving pseudodifferential operators.}
The introduction of the classes $\cM_{\rho,s}(\alpha, \beta)$ is due to the fact that they are closed with respect the  KAM reducibility scheme, for a proper choice of $\alpha$ and $\beta$.
In the next lemma we show how parity preserving pseudodifferential operators embed in such classes.
\begin{lem}[Embedding]\label{lem:emb}
	Given $\alpha, \beta, \rho >0$, consider 
	$F \in \cP\cA^{-\alpha}_\rho$ and $G \in \cP\cA^{-\beta}_\rho$. Assume that
	$$
	F^* = F \ , \qquad G^* = \bar G \ ,
	$$
	(where the adjoint is with respect to the scalar product of $\cH^0$). Define the operator matrix
	\begin{equation}\label{mat:emb}
		\bA :=\left( \begin{matrix}
		F & G \\ -\bar G & -\bar{F} 
		\end{matrix} \right) \ . 
	\end{equation}
	Then, for any $s \geq  0$ and $0<\rho' < \rho$, one has $\bA \in \cM_{\rho', s}(\alpha,  \beta)$. Moreover, there exist $C,c>0$ such that 
	\begin{equation}\label{eq:emb}
		\abs{\bA}_{\rho', s}^{\alpha, \beta} \leq \frac{C}{(\rho-\rho')^\nu}\left( \wp_{s+c}^{-\alpha,\rho}(F)+\wp_{s+c}^{-\beta,\rho}(G) \right) .
	\end{equation}
	Finally, if $F \in \lip_\tw(\Omega, \cP\cA_\rho^{- \alpha})$, $G \in \lip_\tw(\Omega,  \cP\cA_\rho^{-\beta})$, one has 
	$\bA \in \lip_\tw(\Omega, \cM_{\rho',s}(\alpha, \beta))$ and \eqref{eq:emb} holds with the corresponding  weighted Lipschitz norms.
\end{lem}
The proof is available in Appendix \ref{app:tr}.

\paragraph*{Commutators and flows.}
These classes of matrices enjoy also closure properties under commutators and flow generation.
We define the adjoint operator
\begin{equation}
\label{adj}
\ad_\bX(\bV):=\im [\bX,  \bV]  \ ;
\end{equation}
note the multiplication by the imaginary unit in the definition of the adjoint map.
\begin{lem}[Commutator]
\label{lem:com}
Let $\alpha, \rho >0$ and $s > \frac{1}{2}$. 	Assume $\bV \in \cM_{\rho,s}(\alpha,0)$ and $\bX \in \cM_{\rho,s}(\alpha, \alpha)$.  Then $\ad_\bX(\bV)$  
	 belongs to $\cM_{\rho,s}(\alpha, \alpha)$ with the quantitative bound
\begin{equation}
	\label{s.est}
		\Big|\, \ad_\bX(\bV) \Big|_{\rho, s}^{\alpha, \alpha}    \leq 2 \,C_s  \abs{\bX}_{\rho, s}^{\alpha, \alpha} \, \abs{\bV}_{\rho, s}^{\alpha, 0}  \ ;
	\end{equation}
	here $C_s$ is the algebra constant of \eqref{eq:sdecay_norm}.
	Moreover, if $\bV \in \lip_\tw(\Omega,\cM_{ \rho, s}(\alpha,0))$ and $\bX\in\lip_\tw(\Omega,\cM_{ \rho, s}(\alpha,\alpha))$, then $\ad_\bX(\bV)\in\lip_\tw(\Omega,\cM_{ \rho, s}(\alpha,\alpha))$, with 
	\begin{equation}\label{eq:ad_lip_alg}
		\abs{\ad_\bX(\bV)}_{\rho,s,\alpha,\alpha,\Omega}^\wlip{\tw}\leq 2\,C_s  \abs \bX_{\rho, s,\alpha,\alpha,\Omega}^\wlip{\tw}\abs\bV_{\rho, s,\alpha,0,\Omega}^\wlip{\tw} \ .
	\end{equation}
\end{lem} 
Also the proof of this lemma is postponed to Appendix \ref{app:tr}.
\begin{lem}[Flow]\label{lem:flow}
	Let $\alpha, \rho > 0$, $ s> \frac{1}{2}$. Assume $\bV \in \cM_{\rho,s}(\alpha,0)$, $\bX \in \cM_{\rho,s}(\alpha, \alpha)$.  Then the followings hold true:
	\begin{itemize}
		\item[(i)]For any $ r\in[0,s]$ and any $\theta \in \T^\nu$, the operator $e^{\im \bX(\theta)} \in \cL(\bH^r)$, with the standard operator norm uniformly bounded in $\theta$;
		\item[(ii)] The operator   $e^{\im  \bX} \, \bV \, e^{-\im  \bX}$ belongs to $\cM_{\rho,s}(\alpha, 0)$, while 
		$e^{\im  \bX} \, \bV \, e^{-\im  \bX} - \bV$ belongs to $\cM_{\rho,s}(\alpha, \alpha)$ with the quantitative bounds:
		\begin{equation}\label{flow.norm}
			\begin{aligned}
			& \abs{e^{ \im  \bX} \, \bV \, e^{-\im  \bX}}_{\rho,s}^{\alpha, 0} \leq 
			e^{2\, C_s \abs{\bX}_{\rho,s}^{\alpha, \alpha}}  \abs{\bV}_{\rho,s}^{\alpha, 0}  \ ; \\
			& \abs{e^{ \im  \bX} \, \bV \, e^{-\im  \bX} - \bV}_{\rho,s}^{\alpha, \alpha} \leq  2\, C_s e^{2\, C_s \abs{\bX}_{\rho,s}^{\alpha, \alpha}}  \,   \abs{\bX}_{\rho, s}^{\alpha, \alpha} \, \abs{\bV}_{\rho,s}^{\alpha, 0} \ .
			\end{aligned}
		\end{equation}
	\end{itemize}
	Analogous assertions hold for $\bV \in \lip_\tw(\Omega,\cM_{\rho,s}(\alpha,0))$ and $\bX \in \lip_\tw(\Omega,\cM_{\rho,s}(\alpha, \alpha))$.
\end{lem}
The proof of this lemma is a standard application of \eqref{s.est} and the remark that the operator norm is controlled by the $\abs{\cdot}_{\rho,s}^{\alpha, \alpha}$-norm (see also Remark \ref{rem:opnorm_vs_sdecay}).

\section{The Magnus normal form}\label{sec:magnus}
To begin with, we recall  the Pauli matrices notation. Let us introduce
\begin{equation}\label{eq:pauli}
	\bsigma_1=\left(\begin{matrix}
	0 & 1 \\ 1 & 0
	\end{matrix}\right), \; \ \ \  \bsigma_2=\left(\begin{matrix}
	0 & -\im \\ \im & 0
	\end{matrix}\right), \; \ \  \  \bsigma_3=\left(\begin{matrix}
	1 & 0 \\ 0 & -1
	\end{matrix}\right) , 
\end{equation}
and, moreover, define 
$$
\bsigma_4:=
\left(\begin{matrix}
	1 & 1 \\ -1 & -1
	\end{matrix}\right) \ , \quad
	\b1:= \left(\begin{matrix}
	1 & 0 \\ 0 & 1
	\end{matrix}\right) , 
	\quad
	\b0 := \left(\begin{matrix}
	0 & 0 \\ 0 & 0
	\end{matrix}\right). 
$$
Using Pauli matrix notation, equation  \eqref{eq:KG_matrix} reads as
\begin{equation}\label{eq:KG_s}
\begin{aligned}
	\im\dot\vf(t)= & \bH(t)\vf(t):=(\bH_0+\bW(\omega t))\vf(t) \ , \\
&	\bH_0:=B\bsigma_3,\ \ \  \bW(\omega t):=\frac{1}{2}\, B^{-1/2}V(\omega t) B^{-1/2} \bsigma_4 \ .
\end{aligned}
\end{equation}
Note that, by assumption  ({\bf V1}), one has $ V \in \cP\cA^0_\rho $ (see Remark \ref{DPA});
therefore the properties of the pseudodifferential calculus and of the associated symbols (see Remarks \ref{rem:comp} and \ref{rem:comp_pp}) imply that 
 \begin{equation}
 \label{ass_magnus}
 B\in\cP\cA^1  \quad \mbox{and} \quad B^{-1/2}VB^{-1/2} \in \cP\cA^{-1}_\rho 
 \end{equation} 
(in case $\tm = 0$, we use Remark \ref{rem:B-1} to define $B^{-1/2}$).
The difficulty in treating equation \eqref{eq:KG_s} is that it is not perturbative in the size of the potential, so standard KAM techniques do not apply directly.\\

To deal with this problem, we perform a change of coordinates, adapted to fast oscillating systems,   which  puts \eqref{eq:KG_s} in a perturbative setting.
We refer to this procedure as Magnus normal form.
The Magnus normal form is achieved in the following way: the change of coordinates 
  $\vf(t) = e^{- \im \bX(\omega t;\omega)} \psi(t)$ conjugates \eqref{eq:KG_s} to  $\im \d_t \psi(t) = \widetilde{\bH}(t) \psi(t)$, where the Hamiltonian $\widetilde{\bH}(t)$ is given by
(see \cite[Lemma 3.2]{Bam16I}) 
	\begin{align}
		\label{eq:magnus_1}
		\widetilde{\bH}(t) & = \LieTr{\bX(\omega t;\omega)}{\bH(t)}-\int_{0}^1\LieTr{s\bX(\omega t;\omega)}{\dot\bX(\omega t;\omega)}\wrt s \\
		\label{eq:magnus_10}
		& = \bH_0 +\im[\bX,\bH_0]+\bW -\dot\bX + \im [\bX, \ldots] \ .
	\end{align}
In \eqref{eq:magnus_10}	 we wrote, informally, $[\bX, \ldots]$ to remark that all the non written terms are commutators with $\bX$.
Then one chooses $\bX$ to solve $\bW - \dot \bX = 0$; 
if the frequency $\omega$ is large and nonresonant, then   $\bX$ has size $|\omega|^{-1}$, and the new equation \eqref{eq:magnus_10} is now  perturbative in size.
 The price to pay is the appearance of  
 $\im[\bX,\bH_0]$, which is small in size but  possibly unbounded as operator. We control this term by employing pseudodifferential calculus and the properties of the commutators.
 
With this informal introduction, the main result of the section is the following:
\begin{thm}[Magnus normal form]\label{lem:magnus}
For any $0 <\gamma_0 <1$, there exist a set $\Omega_0 \subset R_\tM \subset \R^\nu$ and a constant $c_0 >0$ (independent of $\tM$),  with
 \begin{equation}
\label{meas_omega0}
\frac{\meas (R_{\tM}\backslash\Omega_0)}{\meas(R_\tM)} \leq c_0 \gamma_0 , 
\end{equation}
such that the following holds true. For any $\omega \in \Omega_0$ and  any weight $\tw > 0$, there exists a  time dependent change of coordinates $\vf(t) = e^{- \im \bX(\omega t;\omega)} \psi(t)$, where 
$$\bX(\omega t;\omega)=X(\omega t;\omega)\bsigma_4 \ , \qquad X \in \lip_\tw(\Omega_0,\cP\cA^{-1}_{\rho/2}) \ , $$  that conjugates equation \eqref{eq:KG_s} to 
\begin{equation}\label{eq:system_kg_pert}
	\im \dot\psi(t)=\widetilde{\bH}(t)\psi(t), \quad \widetilde{\bH}(t):=\bH_0+\bV(\omega t; \omega)  \ ,
\end{equation}
where 
\begin{equation}
\label{eq:V}
 \bV(\theta; \omega)=\left( \begin{matrix}
	V^{d}(\theta; \omega) & V^{o}(\theta; \omega)\\ -\oV^{o}(\theta; \omega) & -\oV^{d}(\theta; \omega)
	\end{matrix} \right), \quad {\rm with }  \ \ \ 
	[V^d]^* = V^d  \ , \ \ [V^o]^* = \bar{V^o}
\end{equation}
and
\begin{equation}
\label{VdVo}
	 V^d \in \lip_\tw(\Omega_0,\cP\cA^{-1}_{\rho/2}) \ , \quad V^o\in\lip_\tw(\Omega_0,\cP\cA^0_{\rho/2}) \ .
\end{equation}
Furthermore, for any $\ell \in \N_0$, there exists $C_\ell>0$ such that 
\begin{equation}
\label{est:VdVo}
\wp^{-1, \rho/2}_\ell (V^d)^\wlip{\tw}_{\Omega_0} + \wp^{0, \rho/2}_\ell (V^o)^\wlip{\tw}_{\Omega_0} \leq \frac{C_\ell}{\tM}  .
\end{equation}
\end{thm}
\begin{proof}
	The proof is splitted into two parts, one for the formal algebraic construction, the other for checking that the operators that we have found possess the right pseudodifferential properties we are looking for.\\
	\textbf{Step I).} 
	 Expanding \eqref{eq:magnus_1} in commutators we have
	\begin{equation}\label{eq:magnus_2}
		\widetilde{\bH}(t) = \bH_0 +\im[\bX,\bH_0]-\tfrac{1}{2}[\bX,[\bX,\bH_0]]+\bW -\dot\bX +\bR \ ,
	\end{equation}
		where the remainder $\bR$ of the expansion is given in integral form by
	\begin{equation}\label{eq:magnus_rem}
		\begin{split}
		\bR   := & \int_{0}^1\frac{(1-s)^2}{2}\LieTr{s\bX}{\ad_\bX^3(\bH_0)}\wrt s\\ & + \im \int_0^1\LieTr{s\bX}{[\bX,\bW]}\wrt s -\im\int_0^1(1-s)\LieTr{s\bX}{[\bX,\dot\bX]}\wrt s . 
		\end{split}
	\end{equation}
	From the properties of the Pauli matrices, we note that $\bsigma_4^2=\b0$. This means that the terms in \eqref{eq:magnus_rem} involving $\bW$ and $\dot\bX$ are null, and the remainder is given only by
	\begin{equation}\label{eq:magnus_rem2}
		\bR = \int_{0}^1\frac{(1-s)^2}{2}\LieTr{s\bX}{\ad_\bX^3(\bH_0)}\wrt s . 
	\end{equation}
	We ask $\bX$ to solve the homological equation 
	\begin{equation}\label{hom.eq}
		\b0 = \bW-\dot\bX= \left(\frac{1}{2}\,B^{-1/2} V(\omega t) B^{-1/2} - \dot X(\omega t; \omega) \right) \bsigma_4 .
	\end{equation}	
	Expanding in Fourier coefficients with respect to the angles, its solution is actually given by
	\begin{equation}\label{eq:magnus_sol_hom}
	\begin{aligned}
		&\widehat{X}(k; \omega)=\frac{1}{2\im\,\omega\cdot k}B^{-1/2}\widehat{V}(k) B^{-1/2} , \qquad \mbox{ for } k\in\Z^{\nu}\backslash\{0\}, \\
		& \widehat{X}(0; \omega)\equiv \b0
		\end{aligned}
	\end{equation}
	where the second of \eqref{eq:magnus_sol_hom} is a consequence of ({\bf V2}).
	It remains to compute the terms in \eqref{eq:magnus_1} and \eqref{eq:magnus_rem2} involving $\bH_0$. Using again the structure of the Pauli matrices, we get:
	\begin{equation}\label{eq:ad_X1}
		\ad_\bX(\bH_0):= \im [X\bsigma_4,B\bsigma_3]  = \im XB(\b1-\bsigma_1)-\im BX(\b1+\bsigma_1) = \im[X,B]\b1-\im [X,B]_{\rm a}\bsigma_1 \ ,
	\end{equation}
	where we have denoted by $[X,B]_{\rm a}:= XB + BX$ the anticommutator. Similarly one has 
	\begin{equation}\label{eq:ad_X2}
		\begin{split}
		\ad_\bX^2(\bH_0) & := - [X\bsigma_4,[X\bsigma_4,B\bsigma_3]] \\ 
		& \stackrel{\eqref{eq:ad_X1}}{=} -\left([X\bsigma_4,[X,B]\b1]-[X\bsigma_4,[X,B]_{\rm a}\bsigma_1]\right)\\
		& = -([X,[X,B]]-[X,[X,B]_{\rm a}]_{\rm a})\bsigma_4\\
		& = 4XBX\bsigma_4  \ ; 
		\end{split}
	\end{equation}
	thus
	\begin{equation}
		\begin{split}
		\ad_\bX^3(\bH_0) 
		& \stackrel{\eqref{eq:ad_X2}}{=} 4\im[X\bsigma_4,XBX\bsigma_4]=\b0 \ . 
		\end{split}
	\end{equation}
	This shows that $\bR\equiv \b0$ and, imposing \eqref{eq:magnus_sol_hom} in \eqref{eq:magnus_1}, we obtain
	\begin{equation}\label{eq:magnus_3}
		\widetilde{\bH}(t) = \bH_0 +\bV(\omega t;\omega) \ ,
	\end{equation}
	with
	\begin{equation}\label{eq:magnus_4}
		\begin{split}
			V^d(\theta; \omega) &:= \im[X(\theta; \omega),B]+2X(\theta; \omega)BX(\theta; \omega) \ ,\\
			V^o(\theta; \omega ) &:= -\im[X(\theta; \omega ),B]_{\rm a}+2X(\theta; \omega)BX(\theta; \omega) \ .
		\end{split}
	\end{equation}
	\textbf{Step II).} We show now that $X, V^d$ and $V^o$, defined in \eqref{eq:magnus_sol_hom} and \eqref{eq:magnus_4} respectively, are pseudodifferential operators in the proper classes, provided $\omega$ is sufficiently nonresonant.
	First  consider $\bX$.  For $\gamma_0>0$ and $ \tau_0 >\nu -1$, define the set of Diophantine frequency vectors
	\begin{equation}\label{eq:magnus_omega}
		\Omega_0 \equiv \Omega_0(\gamma_0, \tau_0) :=\Set{\omega\in R_{\tM} | \abs{\omega\cdot k}\geq \frac{\gamma_0}{\braket{k}^{\tau_0}}\tM \quad \forall\,k\in\Z^{\nu}\backslash\{0\} } \ .
	\end{equation}
	We will prove in Proposition \ref{prop:diop_magnus} below that 
	\begin{equation}
	\label{omega0est}
	{\frac{\meas (R_{\tM}\backslash\Omega_0)}{\meas(R_\tM)} \leq c_0 \gamma_0} \ 
	\end{equation}
	for  some constant $c_0>0$  independent of $\tM$ and $\gamma_0$.
 This fixes the set $\Omega_0$ and proves  \eqref{meas_omega0}.\\
	We show now that   $X \in \lip_\tw(\Omega_0,\cP\cA^{-1}_{\rho/2})$. 
	First note that, by Lemma \ref{lem:fou.symb}(i) (in Appendix \ref{app:tr}) and Remark \ref{rem:comp_pp}, one has $B^{-1/2} \wh V(k) B^{-1/2} \in \cP\cA^{-1}$ (both $B$ and $V$ are independent from $\omega$) with  
	$$
	\wp^{-1}_\ell(B^{-1/2} \wh V(k) B^{-1/2}) \leq 4 e^{- \rho|k|} \, \wp^{-1,\rho}_\ell(B^{-1/2} V B^{-1/2}) \leq 4  e^{- \rho|k|} \, C_\ell  .
	$$
	Provided $\omega \in \Omega_0$, it follows that 
	$$
	\wp^{-1}_\ell (\wh X(k; \cdot))^\infty_{\Omega_0} 
	\leq 
	\frac{1}{2}\left[\sup_{\omega \in \Omega_0}\frac{1}{|\omega \cdot k|} \right]\, \wp^{-1}_\ell(B^{-1/2} \wh V(k) B^{-1/2}) 
	\leq \frac{4\la k \ra^{\tau_0}}{\gamma_0 \, \tM } e^{-\rho|k|} C_\ell  .
	$$
	To compute the Lipschitz norm, it is convenient to use the   notation
	\begin{equation}\label{eq:not_lip_easy}
		\Delta_{\omega}f(\omega)	= f(\omega + \Delta \omega)-f(\omega) \ , 
	\end{equation}
	with $\omega$, $\omega+\Delta\omega\in\Omega_0$, $\Delta\omega\neq 0$. In this way one gets
	$$ \abs{\Delta_{\omega}\wh X(k; \omega)} \leq \frac{\abs{\Delta\omega}}{2\abs{\omega\cdot k}\abs{(\omega+\Delta\omega)\cdot k}}\abs{B^{-1/2} \wh V(k) B^{-1/2}} \ \Rightarrow \ \wp_{\ell}^{-1}(\wh X(k; \cdot))^\lip_{\Omega_0} \leq \frac{4\braket{k}^{2\tau_0}}{(\gamma_0\tM)^2}e^{-\rho\abs k}C_\ell \ . $$
	As a consequence  $X(\theta; \omega)= \sum_{k} \wh X(k;\omega) e^{\im k \cdot \theta}$ is a pseudodifferential operator in the class $\lip_\tw(\Omega_0,\cP\cA^{-1}_{\rho/2})$ (see Lemma \ref{lem:fou.symb}(ii) in Appendix \ref{app:tr} for details) fulfilling
	\begin{equation}
	\label{est:X}
	\wp^{-1, \rho/2}_\ell (X)_{\Omega_0}^\wlip{\tw}\leq \left( \frac{1}{\gamma_0\tM}+\frac{\tw}{\gamma_0^2\tM^2} \right)\frac{C_\ell}{\rho^{2\tau_0+\nu}} \leq  \frac{\max(1, \tw)}{\tM} \, \frac{\wt C_\ell}{\rho^{2\tau_0 + \nu}} \ .
	\end{equation}
	It follows by Remark \ref{rem:comp_pp} that $V^d \in \lip_\tw(\Omega_0,\cP\cA^{-1}_{\rho/2})$ while $V^o \in \lip_\tw(\Omega_0,\cP\cA^{0}_{\rho/2})$ with the claimed estimates \eqref{est:VdVo}.\\
	Finally, $V$ is a real selfadjoint operator, simply because it is a real bounded potential, and therefore  $V^* = V = \bar{V}$. 
	It follows by Remark \ref{rem:coef.mat} and 
	the explicit expression \eqref{eq:magnus_sol_hom} that $X^* = X = \bar{X}$.
	Using these properties one verifies by a direct computation that $[V^d]^* = V^d$ and $[V^o]^* = V^o$.
	Estimate \eqref{est:X} and the symbolic calculus of Remark \ref{rem:comp_pp} give \eqref{est:VdVo}.
\end{proof}

\begin{rem}
Everything works with the more general assumptions 
 $V \in \cP\cA^0_\rho$.
\end{rem}

\begin{rem}
\label{rem:p}
	Pseudodifferential calculus is used to guarantee that $V^d$ has order -1 while $V^o$ has order 0 (see \eqref{VdVo}).
	Without this information it would be problematic to apply the standard KAM iteration of Kuksin \cite{kuk87},
	which requires the  eigenvalues to have an asymptotic
	of the form $j + \cO(j^\delta)$ with $\delta < 0$. 
In principle one might circumvent this problem by using the ideas of \cite{BBM14, FP15}
to regularize the order of the perturbation.  
However in our context this smoothing procedure is tricky, since  it produces terms of size $|\omega|$,   which are very large and therefore  unacceptable for our purposes.
\end{rem}

\begin{prop}\label{prop:diop_magnus}
For $\gamma_0>0$ and $ \tau_0 >\nu -1$, 	 the set $\Omega_0$ defined in \eqref{eq:magnus_omega} fulfills  \eqref{omega0est}.
\end{prop}
\begin{proof}
	For any $k \in \Z^\nu \setminus\{0\}$, define  the sets
		$\cG^k:=\Set{ \omega\in R_{\tM} | \abs{\omega\cdot k}< \frac{\gamma_0}{\braket{k}^{\tau_0}}\tM  }$.
	By Lemma \ref{tecnico}
$		\abs{\cG^k}\lesssim
\frac{\gamma_0}{\abs{k}^{\tau_0+1}}\tM^{\nu}$.
	Therefore the set  $\cG:=\bigcup_{k\neq 0}\cG^k$ has measure bounded by $\abs \cG \leq C\gamma_0\tM^{\nu}$, which proves the claim.
\end{proof}

\section{Balanced unperturbed Melnikov conditions}
\label{sub:um}
As we shall see, in order to perform a converging KAM scheme, we must be able to impose  second order Melnikov conditions, namely bounds from below of quantities like
$\omega \cdot k + \lambda_i \pm \lambda_j$, 
where the $\lambda_j$'s are the eigenvalues of the operator $B$ defined in \eqref{def:B}. Explicitly,
\begin{equation}
		\lambda_j:=\sqrt{j^2+\tm^2}=j+\frac{c_j(\tm)}{j} \ , \qquad c_j(\tm) := j (\sqrt{j^2+\tm^2}-j).
\end{equation}
One can check that
$0 \leq c_j(\tm) \leq \tm^2$ $\, \forall j \in \N$.
We introduce the notation of the indexes sets:
\begin{equation}\label{eq:indices_meln}
	\cI^+ :=\Z^{\nu}\times\N\times\N  \ , \qquad 	\cI^-  :=\Set{(k,j,l)\in\cI^+ | (k,j,l)\neq(0,a,a), \ a\in\N} \ .
\end{equation}
Furthermore, we define the relative measure of a measurable set $\Omega$ as
\begin{equation}
\label{def:mr}
\me_r (\Omega) := \frac{|\Omega|}{|R_{\tM}|} \equiv \frac{|\Omega|}{\tM^\nu \, (2^{\nu }-1) c_\nu} 
\end{equation}
where $| \cC|$ is the Lebesgue measure of the set $\cC$ and $c_\nu$ is the volume of the unitary ball in $\R^\nu$.\\
The main result of this section is the following theorem.
\begin{thm}[Balanced Melnikov conditions]\label{lem:meln_unpert}
	Fix  $0 \leq \alpha \leq 1$ and assume that $\tM \geq \tM_0 := \min \{ \tm^2 , \braket\tm^{1/\alpha} \}$ if $\alpha\in [0,1]$.
	Then, for $0<\wt\gamma\leq \min\{ \gamma_0^{3/2},1/8 \}$ and $ \widetilde\tau \geq 2\nu +3$, the set
	\begin{equation}\label{eq:U_alpha}
		\cU_\alpha:=\Set{\omega \in\Omega_0 | \abs{\omega\cdot k+\lambda_j\pm\lambda_l} \geq \frac{\widetilde{\gamma}}{\braket{k}^{\widetilde\tau}}\frac{\braket{j\pm l}^{\alpha}}{\tM^{\alpha}} \ \ \forall (k,j,l)\in\cI^{\pm}  }
	\end{equation}
	is of large relative measure, that is \begin{equation}
	\label{me.rel}
	\me_r(\Omega_0\backslash\cU_\alpha) \leq C \,  \widetilde{\gamma}^{1/3} , 
	\end{equation} where $C>0$ is independent of $\tM$ and $\widetilde \gamma$.
\end{thm}
We will use several times the following standard estimate.
%

\begin{lem}\label{tecnico}
Fix  $k \in \Z^\nu \setminus \{0\}$ and 
let $R_\tM \ni \omega \mapsto \varsigma(\omega) \in \R$ be a Lipschitz function fulfilling
$\abs{\varsigma}^{\lip}_{R_\tM} \leq \tc_0 < |k| . $
Define 
$f(\omega) = \omega \cdot k  + \varsigma(\omega)$.
	Then, for any $\delta \geq 0$, the measure of the set $ A:=\Set{\omega\in R_\tM | \abs{f(\omega)} \leq \delta } $ satisfies the upper bound
	 \begin{equation}\label{measd}
	\abs{A}\leq \frac{ 2 \delta }{|k|- \tc_0} (4\tM)^{\nu-1} \ .
	\end{equation}
	  
\end{lem}
\begin{proof}
Take $\omega_1 = \omega + \epsilon k$, with $\epsilon$ sufficiently small so that $\omega_1 \in R_\tM$. \\
Then
$
\displaystyle{\frac{|f(\omega_1) - f(\omega)|}{|\omega_1 - \omega|} \geq |k| -\abs{\varsigma}^{\lip}_{R_\tM} > |k|- \tc_0}
$
and the estimate follows by Fubini theorem.
\end{proof}
In the rest of the section we write $a \lesssim b$, meaning that $a \leq C b$ for some numerical constant $C>0$ independent of the relevant parameters.

The result of Theorem \ref{lem:meln_unpert} is carried out in two steps. The first one is the following lemma.
\begin{lem}
\label{lem:u1}
Fix  $0 \leq \alpha \leq 1$. 
 There exist $\widetilde\gamma_1>0$ and $\tau_1>\nu + \alpha$ such that the set
	\begin{equation}\label{eq:integer_cond}
		\cT_1:=\Set{\omega\in\Omega_0 | \abs{\omega\cdot k + l}\geq \frac{\widetilde\gamma_1}{\braket{k}^{\tau_1}}\frac{\la{l}\ra^{\alpha}}{\tM^{\alpha}} \quad \forall (k,l)\in\Z^{\nu+1}\backslash\{0\}  }
	\end{equation} 
	has relative measure $\me_r(\Omega_0 \backslash\cT_1) \leq C_1 \widetilde{\gamma_1}$, where $C_1>0$ is independent of $\tM$ and $\widetilde \gamma_1$.
	\end{lem}
\begin{proof}
	If $k=0$ and $l \neq 0$, the estimate in \eqref{eq:integer_cond} holds. The same is true if $k \neq 0$ and $l=0$. Therefore, let  both $k$ and $l$ be different from zero.
	For $\abs l > 4\tM \abs{k}$, the inequality in \eqref{eq:integer_cond} holds true taking $\wtgamma_1\leq \frac{1}{2}$. Indeed:
	\begin{equation*}
		\abs{\omega\cdot k+l}\geq \abs{l}-\abs{\omega}\abs{k}\geq \abs{l}-2\tM\abs{k} \geq \frac{\abs{l}}{2}\geq \frac{1}{2}\abs{l}^{\alpha}\geq \frac{\widetilde\gamma_1}{\braket{k}^{\tau_1}\tM^{\alpha}}\abs{l}^{\alpha} \ .
	\end{equation*}
	Then, consider the case $1 \leq \abs l\leq 4\tM\abs{k}$ (so, only a finite number of $l\in\Z\backslash\{0\}$).
	For fixed $k$ and $l$, define the set
	\begin{equation}
		\cG_l^k:=\Set{ \omega\in R_\tM | \abs{\omega\cdot k + l}\leq \frac{\widetilde\gamma_1}{\braket{k}^{\tau_1}}\frac{\abs{l}^{\alpha}}{\tM^{\alpha}} } \ .
	\end{equation}
	By Lemma \ref{tecnico}, the measure of each set can be estimated by
	\begin{equation}\label{eq:bs1_est}
		\abs{\cG_l^k}\lesssim \tM^{\nu-1} \frac{\wtgamma_1}{\braket{k}^{\tau_1}}\frac{\abs l^{\alpha}}{\tM^{\alpha}}\frac{1}{\abs k} \lesssim \wtgamma_1\tM^{\nu-1-\alpha}\frac{\abs l^{\alpha}}{\la k\ra^{\tau_1+1}} \ .
	\end{equation}
	Let 
	$\cG_1:= \Omega_0 \cap \bigcup \Set{\cG_l^k| (k,l)\in\mathbb{Z}^{\nu+1}\backslash\{0\}, \abs l\leq 4\tM\abs{k}} . $
	Then
	\begin{equation}
		\begin{split}
		\abs{\cG_1} & \leq \sum_{k \in \Z^\nu \setminus\{0\}}\sum_{l \in \Z\setminus \{0\} \atop \abs l\leq 4\tM\abs{k}}\abs{\cG_l^k} \stackrel{\eqref{eq:bs1_est}}{\lesssim} \wtgamma_1\tM^{\nu-1-\alpha} \sum_{k\neq 0}\sum_{\abs l\leq 4\tM\abs{k}}\frac{\abs{l}^{\alpha}}{\la{k}\ra^{\tau_1+1}}\\
		&\lesssim \wtgamma_1\tM^{\nu-1-\alpha}  \sum_{k\neq 0}\frac{1}{\la{k}\ra^{\tau_1+1}}(4\tM\abs{k})^{\alpha+1} \lesssim \wtgamma_1 \tM^{\nu} \sum_{k\neq 0}\frac{1}{\la k\ra^{\tau_1-\alpha}}\lesssim \wtgamma_1\tM^{\nu}
		\end{split}
	\end{equation}
	provided $\tau_1> \nu + \alpha$.   It follows that the relative measure of $\cG_1$ is given by
	\begin{equation}
		\me_r(\cG_1)\leq C_1 \wtgamma_1 \ ,
	\end{equation}
	where $C_1>0$ is independent of $\tM$ and $\wtgamma_1$. The thesis follows, since $\cT_1 = \Omega_0 \setminus \cG_1$.
\end{proof}
	
\begin{rem}
In case $\tm =0$, Lemma \ref{lem:u1} implies Theorem \ref{lem:meln_unpert}.
\end{rem}	
From now on assume that $\tm > 0$. The second step is the next lemma.

\begin{lem}\label{lem:new_T2}
	There exist $0<\wtgamma_2\leq \min\{ \gamma_0,\wtgamma_1/2 \}$ and $\tau_2\geq \tau_1+\nu+1$ such that the set
	\begin{equation}\label{eq:T2}
		\cT_2:=\Set{\omega \in \cT_1 | \abs{\omega\cdot k +\lambda_j\pm \lambda_l} \geq \frac{\wtgamma_2}{\braket{k}^{\tau_2}}\frac{\braket{j\pm l}^\alpha}{\tM^\alpha} \quad \forall (k,j,l)\in\cI^{\pm} }
	\end{equation} 
	 fulfills $\displaystyle{\me_r(\cT_1\setminus\cT_2)\leq C_2 \frac{\wtgamma_2}{\wtgamma_1}}$, where $C_2>0$ is independent of $\tM$, $\wtgamma_1$, $\wtgamma_2$.
\end{lem}
\begin{proof}
	Let $(k,j,l)\in\cI^{\pm}$. We can rule out some cases for which the inequality in \eqref{eq:T2} is already satisfied when $\omega\in\cT_1\subset\Omega_0$:
	\begin{itemize}
		\item For $\pm=+$ and $k=0$, we have
		$$ \lambda_j+\lambda_l = j+l +\frac{c_j(\tm)}{j}+\frac{c_l(\tm)}{l}\geq j+l\geq \frac{\wtgamma_2}{\tM^\alpha}\braket{j+l}^\alpha \  ; $$
		\item For $\pm=-$ and $k\neq 0 $, $j=l$,
	we have 	$\displaystyle{ \abs{\omega\cdot k}\geq \frac{\gamma_0}{\braket{k}^{\tau_0}}\tM}$;
		\item For $\pm=-$ and $k=0$, $j\neq l$, and $\alpha \in (0,1]$, it holds that
		$$ \abs{\lambda_j-\lambda_l} =
		\abs{\int_{l}^{j}\frac{x}{\sqrt{x^2+\tm^2}}\wrt x}
		\geq  
		\frac{1}{\braket{\tm}}\abs{l-j}  \geq \frac{\wtgamma_2}{\tM^\alpha}\braket{j-l}^\alpha \ . $$
		For  $\alpha=0$ the estimate is trivially verified.
	\end{itemize}
	Therefore, for the rest of this argument, let $k\neq 0$ and $j\neq l$. Assume first that $\abs{j\pm l}\geq 8\tM\abs k$. In this case, one has:
	\begin{equation*}
		\begin{split}
		\abs{\omega\cdot k+\lambda_j\pm \lambda_l} & \geq \abs{j\pm l} - \abs{ \frac{c_j(\tm)}{j}\pm\frac{c_l(\tm)}{l} } - \abs{\omega\cdot k} {\geq} \abs{j \pm l} -4\tM\abs k \geq \frac{1}{2}\abs{j\pm l} \ .
		\end{split}
	\end{equation*}
	Let now $\abs{j\pm l}< 8\tM\abs k$. In the region $j< l$ assume
	\begin{equation}\label{eq:Rk}
		j\braket{j\pm l}^\alpha \geq \tR(k):= \frac{4\tm^2 \tM^\alpha\braket{k}^{\tau_1}}{\wtgamma_1} \ ,
	\end{equation}
	where $\wtgamma_1$ and $\tau_1$ are the ones of Lemma \ref{lem:u1}.
	So, for $\omega\in \cT_1$, we get
	\begin{equation}
		\begin{split}
		\abs{\omega\cdot k+\lambda_j\pm \lambda_l} & \geq \abs{\omega\cdot k + l\pm j} - \abs{ \frac{c_j(\tm)}{j}\pm\frac{c_l(\tm)}{l} }\\
		& \geq \frac{\wtgamma_1}{\braket{k}^{\tau_1}} \frac{\braket{j\pm l}^\alpha}{\tM^\alpha}-\frac{2 \tm^2}{j} \stackrel{(\ref{eq:Rk})}{\geq}  \frac{\wtgamma_1}{2\braket{k}^{\tau_1}} \frac{\braket{j\pm l}^\alpha}{\tM^\alpha} .
		\end{split}
	\end{equation}
	Thus, we consider just those $j$ and $l$ with  $j\braket{j\pm l}^\alpha< \tR(k)$. The symmetric argument shows that we can take those $l< j$ for which $l\braket{j\pm l}^\alpha<\tR(k)$.\\
	Like in the previous proof, consider the set
	\begin{equation}
		\cG_{j,l}^{k,\pm}:=\Set{\omega\in R_\tM | \abs{\omega\cdot k +\lambda_j\pm \lambda_l} < \frac{\wtgamma_2}{\braket{k}^{\tau_2}}\frac{\braket{j\pm l}^\alpha}{\tM^\alpha} }
	\end{equation}
	defined for those $k\neq 0$ and $j\neq l$ in the regions
	\begin{equation}
		\cP^{\pm}:= \set{ \abs{j\pm l}< 8\tM\abs k } \cap \Big( \set{j\braket{j\pm l}^\alpha< \tR(k), \ j< l} \cup \set{l\braket{j\pm l}^\alpha< \tR(k), \ l<j} \Big) \ .
	\end{equation}
	Using Lemma \ref{tecnico}, the estimate for its Lebesgue measure is
	\begin{equation}\label{eq:bs2_rev}
		\abs{\cG_{j,l}^{k,\pm}}\lesssim \wtgamma_2 \tM^{\nu-1-\alpha} \frac{\braket{j\pm l}^\alpha}{\abs k^{\tau_2+1}} \ .
	\end{equation}
	Define $ \cG_2^\pm:= \cT_1\cap \bigcup \Set{\cG_{j,l}^{k,\pm} |  (k,j,l)\in\cP^\pm } $. By symmetry of the summand, we estimate
	\begin{equation}
		\begin{split}
		\abs{\cG_2^-} & \leq \sum_{(k,j,l)\in\cP^-} \abs{\cG_{j,l}^{k,-}} \stackrel{\eqref{eq:bs2_rev}}{\lesssim} \wtgamma_2\tM^{\nu-1-\alpha}\sum_{(k,j,l)\in\cP^-}\frac{\braket{j- l}^\alpha}{\abs k^{\tau_2+1}} \\
		& \lesssim  \wtgamma_2\tM^{\nu-1-\alpha}\sum_{k\neq 0}\sum_{j<l \atop  j\braket{j-l}^\alpha<\tR(k)}\sum_{\abs{j-l}<8\tM\abs k }\frac{\braket{j-l}^\alpha}{\abs k^{\tau_2+1}}\\
		& \lesssim  \wtgamma_2\tM^{\nu-1-\alpha}\sum_{k\neq 0} \ \ \sum_{l- j=:h>0 \atop  h <8\tM\abs k} \ \ \sum_{j <\tR(k) \braket{h}^{-\alpha}}\frac{\braket{h}^\alpha}{\abs k^{\tau_2+1}}\\
		& \stackrel{\eqref{eq:Rk}}{\lesssim} \frac{\wtgamma_2}{\wtgamma_1}\tM^{\nu-1}\sum_{k\neq 0}\sum_{ h <8\tM\abs k}\frac{1}{\abs k^{\tau_2+1-\tau_1}} \lesssim \frac{\wtgamma_2}{\wtgamma_1}\tM^{\nu}\sum_{k\neq 0}\frac{1}{\abs k^{\tau_2-\tau_1}} \leq \frac{\wtgamma_2}{\wtgamma_1}\tM^{\nu}
		\end{split}
	\end{equation}
	provided $\tau_2> \tau_1 + \nu$. The same computation holds for $\cG_2^+$. We conclude that
	\begin{equation}
	\me_r(\cT_1\setminus\cT_2) \leq \me_r(\cG_2^-\cap \cG_2^+) \leq C_2 \frac{\wtgamma_2}{\wtgamma_1} \ ,
	\end{equation}
	where $C_2>0$ is independent of $\tM$, $\wtgamma_1, \wtgamma_2$.
\end{proof}

\begin{proof}[Proof of Theorem \ref{lem:meln_unpert}]
	Take 
	$\wtgamma_1 = \wtgamma^{1/3}$, $ \wtgamma_2 = \wtgamma^{2/3}  $
	with some $\wtgamma >0$ sufficiently small so that $\wtgamma_1$ and $ \wtgamma_2$ fulfill the assumptions of the previous lemmas. Similarly, choose 
$\tau_1 = \nu +2$ and $ \tau_2= 2\nu +3 \ . $
	By definition, $\cU_\alpha \equiv \cT_2 \subset \Omega_0$. Since $ \Omega_0 \setminus \cU_\alpha = (\Omega_0 \setminus \cT_1) \cup (\cT_1 \setminus \cT_2) $, we get by Lemma \ref{lem:u1} and Lemma \ref{lem:new_T2} that
	$$
	\me_r(\Omega_0 \setminus \cU_\alpha)\leq  C_1  \wtgamma_1 + C_2 \frac{\wtgamma_2}{\wtgamma_1} \leq C \wtgamma^{1/3} \ , \qquad C = 2\left( C_1 + C_2  \right) . 
	$$
\end{proof}

\section{The KAM reducibility transformation}\label{sec:kam}
The new potential $\bV(\omega t ; \omega)$ that we have found in Theorem \ref{lem:magnus} is perturbative, in the sense that the smallness of its norm is controlled by the size $\tM$ of the frequency vector $\omega$.
Thus, we are now ready to attack with a KAM reduction scheme, presenting first the algebraic construction of the single iteration, then quantifying it via the norms and seminorms that we have introduced in Section 2. The complete result for this reduction transformation, together with its iterative lemma, is proved at the end of this section. 
\subsection{Preparation for the KAM iteration}

Actually, for the KAM scheme it is more convenient to work with operators of type $\cM_{\rho,s}$. Of course, as we have seen in Section \ref{sec:fun}, pseudodifferential operators analytic in $\theta$ belong to such a class.

\begin{lem}
Fix an arbitrary $s_0 > 1/2$ and put $\rho_0 := \rho/4$. Then the operator $\bV(\omega) $ defined in \eqref{eq:V} belongs to $\lip_\tw(\Omega_0,\cM_{\rho_0, s_0}(1,0))$ with the quantitative bound  
\begin{equation}\label{eq:magnus_size}
	\abs{\bV}_{\rho_0,s_0, 1, 0,\Omega_0}^{\wlip{\tw}} \leq \frac{C}{\tM} ;
\end{equation}
here  $C >0$ is independent of $\tM$.
\end{lem}
\begin{proof}
It is sufficient to apply the embedding Lemma \ref{lem:emb} and \eqref{est:VdVo}.
\end{proof}

\subsection{General step of the reduction}\label{subsec:alg_red}
Consider the system
\begin{equation}\label{eq:system_pert_gen}
\im \dot\psi(t)=\bH(t)\psi(t), \quad \bH(t):=\bA(\omega)+\bP(\omega t; \omega) ,
\end{equation}
where
	 the frequency vector $\omega$ varies in some set $\Omega\subset\R^{\nu}$, $\tM\leq\abs\omega\leq 2\tM $;
	 the time-independent operator $\bA(\omega)$ is diagonal, with
	\begin{equation}
		\bA(\omega)=\left(\begin{matrix}
		A(\omega) & 0 \\ 0 & -A(\omega)
		\end{matrix}\right), \quad A(\omega):=\diag\{ \lambda_j^-(\omega) \,|\, j\in\N \}\subset (0,\infty)^{\N} \ ;
	\end{equation}
	and  the quasi-periodic perturbation $\bP(\omega t; \omega)$ has the form
	\begin{equation}
		\bP(\omega t; \omega) =\left(\begin{matrix}
		P^d(\omega t; \omega) & P^o(\omega t; \omega) \\ -\bar{P^o}(\omega t; \omega) & -\bar{P^d}(\omega t; \omega)
		\end{matrix}\right)  , \quad P^d = [P^d]^* \ , \quad \bar{P^o} = [P^o]^*  \ .
	\end{equation}
The goal is to square the size of the perturbation (see Lemma \ref{lem:est_P+}) and we do it by conjugating the Hamiltonian $\bH(t)$ through a transformation $ \psi:= e^{ -\im \bX^+(\omega t; \omega) }\vf $ of the form
\begin{equation}\label{eq:matrix_transf}
  \bX^+(\omega t; \omega)=\left( \begin{matrix}
X^{d}(\omega t; \omega) & X^{o}(\omega t; \omega) \\ -\overline{X^{o}}(\omega t; \omega) & -\overline{X^{d}}(\omega t; \omega)
\end{matrix} \right) \ , \quad X^d = [X^d]^*  ,   \ \ \overline{X^{o}} = [X^o]^* \ ,
\end{equation}
so that the transformed Hamiltonian, as in \eqref{eq:magnus_1}, is
\begin{equation}
	\bH^+(t):= \LieTr{\bX^+(\omega t;\omega)}{\bH(t)}- \int_{0}^1\LieTr{s\bX^+(\omega t; \omega)}{\dot \bX^+(\omega t;\omega)}\wrt s \ .
\end{equation}
Its expansion in commutators is given by
\begin{equation}\label{eq:transf_ham_1}
\begin{aligned}
\bH^+(t) 
& = \bA +\bP  +\im [\bX^+, \bA]-\dot{\bX}^+ + \bR  ,   \\
 \bR & :=   \LieTr{\bX^+}{\bA} - (\bA + \im[\bX^+, \bA])  + \LieTr{\bX^+}{\bP}-\bP - \left( \int_{0}^{1}\LieTr{s \bX^+}{\dot{\bX}^+}\wrt{s} - \dot{\bX}^+ \right) \ .
\end{aligned}
\end{equation}	
We ask now $\bX^+$ to solve the "quantum" homological equation:
\begin{equation}\label{eq:hom_eq_kam_gen}
\im [\bX^+(\theta), \bA]-\omega \cdot \partial_{\theta}\bX^+(\theta)+\Pi_N\bP(\theta) = \bZ
\end{equation}
where $
\Pi_N\bP(\theta; \omega) := \sum_{|k| \leq N} \wh P(k; \omega) e^{\im k \cdot \theta} 
$ is the projector on the frequencies smaller than $N$, while $\bZ $ is the diagonal, time independent part of $P^d$:
\begin{equation}\label{eq:Z_kam_gen}
	\bZ = \bZ(\omega) := \left(\begin{matrix}
	Z(\omega) & 0\\ 0 & - Z(\omega)
	\end{matrix}\right) \ , \quad  Z = \diag\{\wh{(P^d)_j^j}(0;\omega)\,|\, j\in\N \} \ .
\end{equation}
With this choice,  the new Hamiltonian becomes $\bH(t)^+=\bA^+ + \bP(\omega t)^+$ with
\begin{equation}\label{eq:transf_ham_kam}
	 \bA^+ = \bA + \bZ , \qquad 
		\bP^+  :=
	 \Pi_N^{\perp} \bP + \bR \ , \qquad  \Pi_N^{\perp} \bP:= (\uno - \Pi_N)\bP .
	\end{equation}
In order to solve equation \eqref{eq:hom_eq_kam_gen}, note that it reads  block-wise as
\begin{equation}
\left\{ \begin{matrix}
\im [X^{d},A]-\omega\cdot \partial_{\theta}X^{d}+P^{d} = Z\\
-\im [X^{o},A]_{\rm a}-\omega\cdot \partial_{\theta}X^{o}+P^{o} = 0
\end{matrix} \right. \ .
\end{equation}
Expanding both with respect to the exponential basis of $B$ (for the space) and in Fourier in angles (for the time), we get the solutions
\begin{equation}\label{eq:sol_eq_kam_diag}
	\wh{(X^d)_l^j}(k;\omega):=\left\{ \begin{array}{ll}	\displaystyle{\frac{1}{\im(\omega\cdot k+\lambda_j^-(\omega)-\lambda_l^-(\omega))}\widehat{(P^d)_l^j}(k;\omega)} & (k,j,l)\in\cI_N^- \\
	0 & \text{otherwise}
	\end{array} \right. \ ,
\end{equation}
\begin{equation}\label{eq:sol_eq_kam_anti}
	\wh{(X^o)_l^j}(k;\omega):=\left\{ \begin{array}{ll}
	\displaystyle{\frac{1}{\im(\omega\cdot k+\lambda_j^-(\omega)+\lambda_l^-(\omega))}\widehat{(P^o)_l^j}(k;\omega)} & (k,j,l)\in \cI_N^+ \\
	0 & \text{otherwise}
	\end{array} \right. \ ,
\end{equation}
where, following the notation in \eqref{eq:indices_meln}, we have defined
\begin{equation}
\label{IN}
	\cI_N^\pm:=\Set{ (k,j,l)\in \cI^\pm | \abs k \leq N } \ .
\end{equation}
Remark that  $A^+=\diag\{ \lambda_j^+(\omega) \,|\, j\in\N \}$ with 
$
	\lambda_j^+(\omega):=\lambda_j^-(\omega)+\wh{(P^d)_j^j}(0;\omega) .$

\subsection{Estimates for the general step}\label{subsec:est_gen_kam}
Both for well-posing the solutions \eqref{eq:sol_eq_kam_diag} and \eqref{eq:sol_eq_kam_anti} and ensuring convergence of the norms, 
second order Melnikov conditions are required to be imposed. In particular, we choose the frequency vector from the following set
\begin{equation}\label{eq:meln_set}
	\Omega^+:=\Set{ \omega \in \Omega | \abs{\omega\cdot k+\lambda_j^-(\omega)\pm \lambda_l^-(\omega)}\geq \frac{\gamma}{2\braket{N}^{\tau}}\frac{\braket{j\pm l}^{\alpha}}{\tM^{\alpha}} \ ,  \  \forall\, (k, j,l) \in \cI^\pm_N}
\end{equation}
with $\gamma, \tau >0$ to be fixed later on. Here $\cI^\pm_N$ has been  defined in \eqref{IN}.\\
The fact that $\Omega^+$ is actually a set of large measure, that is $\me_r(\Omega\backslash\Omega^+)=O(\gamma)$, will be clear as a direct consequence of Lemma  \ref{lem:meas_infty} of Section \ref{sub:iter_kam}.\\
From now on, we choose as Lipschitz weight $\tw:=\gamma/\tM^\alpha$ and, abusing notation,  we denote
\begin{equation*}
	\lip_\gamma(\Omega,\cF):=\lip_{\gamma/\tM^\alpha}(\Omega,\cF) \ .
\end{equation*}
Furthermore, {\bf we fix once for all $s_0 > 1/2$ and $\alpha \in (0,1)$}.\\
 For $\bV\in\lip_\gamma(\Omega,\cM_{ \rho, s_0}(\alpha,0))$, we write
\begin{equation*}
	\begin{aligned}
	\abs\bV & := \abs\bV_{s_0}^{\alpha,0} \ ,  \qquad 
	\abs\bV_\rho  := \abs\bV_{\rho, s_0}^{\alpha,0}  , \qquad 
	\abs\bV_{\rho,\Omega}^\wlip{\gamma}  := \abs\bV_{\rho,s_0,\alpha,0,\Omega}^\wlip{\gamma/\tM^\alpha} \equiv \abs\bV_{\rho,\Omega}^\infty+\frac{\gamma}{\tM^\alpha}\abs\bV_{\rho,\Omega}^\lip \ ,
	\end{aligned} 
\end{equation*}
while, for $\bV\in\lip_\gamma(\Omega,\cM_{ \rho, s_0}(\alpha,\alpha))$, we denote
\begin{equation*}
\begin{aligned}
\opnorm\bV_\rho  := \abs\bV_{\rho, s_0}^{\alpha,\alpha}  \ , \qquad 
\opnorm\bV_{\rho,\Omega}^\wlip{\gamma} & := \abs\bV_{\rho,s_0,\alpha,\alpha,\Omega}^\wlip{\gamma/\tM^\alpha} \equiv \opnorm\bV_{\rho,\Omega}^\infty+\frac{\gamma}{\tM^\alpha}\opnorm\bV_{\rho,\Omega}^\lip \ .
\end{aligned} 
\end{equation*}
\begin{rem}\label{rem:V_vs_V}
Note that $\abs{\bV}_{\rho_0,\Omega_0}^\wlip{\gamma}\leq\norm{\bV}_{\rho_0,\Omega_0}^\wlip{\gamma}$.
\end{rem}
Now, we provide the estimate on the generator $\bX^+$ of the previous transformation.
For sake of simplicity during the forthcoming proof, as short notation we define
\begin{equation}
	\tg_{j,l}^{k,\pm}(\omega):=\omega\cdot k+\lambda_j^-(\omega)\pm \lambda_l^-(\omega)
\end{equation}
for $(k,j,l)\in\cI_N^\pm$.\\
\begin{lem}\label{lem:est_gen_kam}
Assume that:
	\begin{itemize}
		\item[(a)] $\bP\in\lip_\gamma(\Omega,\cM_{ \rho, s_0}(\alpha,0))$, with an arbitrary $\rho>0$;
		\item[(b)] There exists $0<\tC\leq 1$ such that for any $j\in\N$, $\omega,\Delta\omega\in\Omega^+$ one has 
		\begin{equation}\label{eq:ass_small}
		\abs{\Delta_{\omega}\lambda_j^-(\omega)}\leq \tC \abs{\Delta\omega} .
		\end{equation}
		
	\end{itemize}
Let $\bX^+=\bX^+(\omega t, \omega)$ be defined by \eqref{eq:sol_eq_kam_diag} and \eqref{eq:sol_eq_kam_anti}. 	Then $\bX^+\in\lip_\gamma(\Omega^+,\cM_{\rho, s_0}(\alpha,\alpha))$ with the quantitative bound
	\begin{equation}\label{eq:gen_kam_est}
		\opnorm{\bX^+}_{\rho,\Omega^+}^{\wlip{\gamma}}\leq 16 \braket{N}^{2\tau + 1}\frac{\tM^{\alpha}}{\gamma}  \abs \bP_{\rho, \Omega}^{\wlip{\gamma}}    .
	\end{equation}
\end{lem}
\begin{proof}
	 We start with the seminorm $\opnorm{\bX^+}_{\rho,\Omega^+}^{\infty}$. Fix $\omega \in \Omega^+$ and $\abs k \leq N$. Then, when $j\neq l$, we have
	\begin{equation}\label{eq:est_gen_diag}
		\abs{\wh{(X^d)_l^j}(k;\omega)} \leq \frac{1}{\abs{\tg_{j,l}^{k,-}(\omega)}}\abs{\widehat{(P^d)_l^j}(k;\omega)}\leq \frac{2\braket{N}^{\tau}\tM^{\alpha}}{\gamma}\frac{\abs{\widehat{(P^d)_l^j}(k;\omega)}}{\braket{j-l}^{\alpha}}
	\end{equation}
	and similarly, for any $j,l\in\N$
	\begin{equation}\label{eq:est_gen_out}
		\abs{\wh{(X^o)_l^j}(k;\omega)} \leq \frac{2\braket{N}^{\tau}\tM^{\alpha}}{\gamma}\frac{\abs{\widehat{(P^o)_l^j}(k;\omega)}}{\braket{j+l}^{\alpha}} \ .
	\end{equation}
 From the assumption $(a)$, we have that all the terms
		$\abs{\braket{D}^{\alpha}\wh{P^{d}}(k;\omega)}_{s_0}$, $\; \abs{\wh{P^{d}}(k;\omega)\braket{D}^{\alpha}}_{s_0},$  $\abs{\braket{D}^{\sigma}\wh{P^{\delta}}(k;\omega)\braket{D}^{-\sigma}}_{s_0}$
	(with $\sigma=\pm\alpha,0$, $\delta=d,o$) are bounded.
In order to bound $\opnorm{\wh{\bX^+}(k;\omega)}$,  what we have to prove is that we can control also the terms
	$$ \abs{\braket{D}^{\alpha}\wh{X^{\delta}}(k;\omega)}_{s_0}, \ \ \ \abs{\wh{X^{\delta}}(k;\omega)\braket{D}^{\alpha}}_{s_0},\ \ \  \abs{\braket{D}^{\sigma}\wh{X^{\delta}}(k;\omega)\braket{D}^{-\sigma}}_{s_0} \ . $$
	The seminorms involving the diagonal term $X^d$ can be easily handled, since, by \eqref{eq:est_gen_diag}, they are essentially bounded by the same seminorms for $P^d$. 
	The similar bound in \eqref{eq:est_gen_out} is enough also when we consider the terms $\abs{\braket{D}^{\sigma}\wh{X^o}(k;\omega)\braket{D}^{-\sigma}}_{s_0}$. 
Consider now the term 	 $\braket{D}^{\alpha} \wh{X^o}(k;\omega)$. 
	Applying again \eqref{eq:est_gen_out}, we get
	\begin{equation}		
		\begin{split}
		\abs{\left( \braket{D}^{\alpha} \wh{X^o}(k;\omega) \right)_l^j} & = \abs{ \braket{l}^{\alpha}\wh{(X^o)_l^j}(k;\omega) }  \leq \frac{2\braket{N}^{\tau}\tM^{\alpha}}{\gamma}\frac{\braket{l}^{\alpha}}{\braket{j+l}^{\alpha}}\abs{ \wh{(P^o)_l^j}(k;\omega)} \\
		& \leq \frac{2\braket{N}^{\tau}\tM^{\alpha}}{\gamma}\abs{ \wh{(P^o)_l^j}(k;\omega)} \ .
		\end{split}
	\end{equation}
	The same bound holds for $\abs{\left(\wh{X^o}(k;\omega)\braket{D}^{\alpha}\right)^j_l}$. We obtain that
	\begin{equation*}
		\opnorm{\bX^+}_{\rho,\Omega^+}^{\infty} \leq \frac{2\braket{N}^{\tau}\tM^{\alpha}}{\gamma}\abs\bP_{\rho, \Omega}^{\infty} \ .
	\end{equation*}
	We deal now with the estimates on the Lipschitz seminorm $\opnorm{\bX^+}_{\rho,\Omega^+}^{\lip}$. Using the notation \eqref{eq:not_lip_easy} we have, for $\delta=d,o$:
	\begin{equation}\label{eq:sol_lip_kam}
		\begin{split}
		\Delta_{\omega}\wh{(X^{\delta})_l^j}(k;\omega)
		& = -\im\frac{\Delta_{\omega}(\tg_{j,l}^{k,\pm}(\omega))}{\tg_{j,l}^{k,\pm}(\omega+\Delta\omega)\tg_{j,l}^{k,\pm}(\omega)} \wh{(P^{\delta})_l^j}(k;\omega)+\frac{\im}{\tg_{j,l}^{k,\pm}(\omega+\Delta\omega)}\Delta_{\omega}\wh{(P^{\delta})_l^j}(k;\omega) \ .
		\end{split}
	\end{equation}
	By the assumption in \eqref{eq:ass_small}, we have that
	\begin{equation}
		\abs{\Delta_{\omega}(\tg_{j,l}^{k,\pm}(\omega))}=\abs{\Delta\omega\cdot k +  \Delta_{\omega}(\lambda_j^-\pm\lambda_l^-)} \stackrel{\eqref{eq:ass_small}}{\leq} \abs k\abs{\Delta\omega} +2 \tC \abs{\Delta\omega} \leq  \braket{N}\abs{\Delta\omega}
	\end{equation}
	uniformly for every $j,l\in\N$ and $k\in\Z^{\nu}$, $\abs k \leq N$. We can thus estimate \eqref{eq:sol_lip_kam} by
	\begin{equation}\label{eq_sol_lip_kam_est} 
		\begin{split}
		\abs{\Delta_{\omega}\wh{(X^{\delta})_l^j}(k;\omega)} 
		& \leq \frac{8\braket{N}^{2\tau+1}\tM^{2\alpha}\abs{\Delta\omega}}{\gamma^2}\frac{\abs{\wh{(P^{\delta})_l^j}(k;\omega)}}{\braket{j\pm l}^{2\alpha}}+\frac{2\braket{N}^{\tau}\tM^{\alpha}}{\gamma}\frac{\abs{\Delta_{\omega}\wh{(P^{\delta})_l^j}(k;\omega)}}{\braket{j\pm l}^{\alpha}} \ ,
		\end{split}
	\end{equation}
	from which one deduces easily the claimed estimate \eqref{eq:gen_kam_est}.
\end{proof}
\begin{lem}\label{lem:est_P+}
	Let $\bP\in\lip_\gamma(\Omega,\cM_{ \rho, s_0}(\alpha,0))$. Assume \eqref{eq:ass_small} and, for some fixed $C_{s_0}>0$,
	\begin{equation}\label{eq:small_P}
		C_{s_0} \, 16 \braket{N}^{2\tau+1}\frac{\tM^\alpha}{\gamma}\abs\bP_{\rho, \Omega}^\wlip{\gamma}<1 \ .
	\end{equation}
	Then $\bP^+=\Pi_N^\perp\bP+\bR$, defined as in \eqref{eq:transf_ham_kam}, belongs to $\lip_\gamma(\Omega^+,\cM_{\rho^+, s_0}(\alpha,0))$ for any $\rho^+\in(0,\rho)$, with bounds
	\begin{equation}
	\label{est:small_P}
		\abs{\Pi_N^{\perp}\bP}_{\rho^+,\Omega}^{\wlip{\gamma}}\leq e^{-(\rho-\rho^+)N} \abs \bP_{\rho, \Omega}^{\wlip{\gamma}} \ , \quad \opnorm{\bR}_{\rho,\Omega^+}^{\wlip{\gamma}} \leq C_{s_0} \, 2^9 \frac{\tM^\alpha}{\gamma}\braket{N}^{2\tau+1}\left(\abs\bP_{\rho, \Omega}^\wlip{\gamma}\right)^2 \ .
	\end{equation}
\end{lem}
\begin{proof}
The estimate on $\Pi_N^{\perp}\bP$ follows by using  that it contains only high frequencies.
To estimate the remainder $\bR$, use \eqref{eq:transf_ham_1},\eqref{eq:hom_eq_kam_gen} to write it as 
	\begin{equation}
		\bR  =  \int_{0}^1(1-s)\LieTr{s\bX^+}{\ad_{\bX^+}(\bZ-\bP)}\wrt s + \int_{0}^1\LieTr{s\bX^+}{\ad_{\bX^+}(\bP)}\wrt s \ .
	\end{equation}
	Then, apply Lemma \ref{lem:flow} and   Lemma \ref{lem:est_gen_kam}.
\end{proof}
\begin{rem}\label{rem:new_pert_est}
	Defining the quantities
	$$
	\displaystyle{\eta:=\frac{\tM^{\alpha}}{\gamma}\abs\bP_{\rho,\Omega}^{\wlip{\gamma}}} , \qquad \displaystyle{\eta^+:=\frac{\tM^{\alpha}}{\gamma}\abs{\bP^+}_{\rho^+,\Omega^+}^{\wlip{\gamma}}}$$
	and  choosing  $N = -(\rho- \rho^+)^{-1}\ln\eta $, Lemma \ref{lem:est_P+} implies that 
	\begin{equation}
	\eta^+\leq \left( e^{-(\rho-\rho^+) N}+\braket{N}^{2\tau+1} \eta \right)\eta  \leq  \left(1 + \frac{1}{(\rho - \rho^+)^{2\tau+1}}\left(\ln \frac{1}{\eta}\right)^{2\tau+1}\right)\eta^2 \ .
	\end{equation}
\end{rem}

\subsection{Iterative Lemma and KAM reduction}\label{sub:iter_kam}
Once that the general step has been illustrated, we are ready for setting our iterative scheme. The Hamiltonian  the iteration starts with is the one that we have found after the Magnus normal form  in Section \ref{sec:magnus}:
\begin{equation}\label{eq:step0_kam}
	\bH^{(0)}(t) = \bH_0^{(0)} + \bV^{(0)}(\omega t; \omega) \ , \qquad \abs{\bV^{(0)}}_{\rho_0,\Omega_0}^{\wlip{\gamma}}{\leq} \frac{C}{\tM} \ ,
\end{equation}
where $\bH_0^{(0)}:=\bH_0$ and $\bV^{(0)}:=\bV$ as in Theorem \ref{lem:magnus}. All the iterated objects are constructed from the  transformation in Sections \ref{subsec:alg_red}, \ref{subsec:est_gen_kam} by setting for $n\geq 0$
\begin{align*}
	\bH^{(n)}(t):=\bA(\omega)+\bP(\omega t; \omega) \ , \quad \bA :=  \bH_0^{(n)} \ , \quad \bP:=\bV^{(n)} \\
	 \bZ^{(n)}:= \bZ  \ , \quad\bX^{(n)}:=  \bX   \ , \quad \bR^{(n)}:= \bR . 
\end{align*}
Given reals  $\gamma, \rho_0, \eta_0 >0$ and a sequence of nested sets $\{\Omega_n\}_{n \geq 1}$, we  fix the parameters 
$$
\displaystyle{\delta_n := \frac{3}{\pi^2 (1+n^2)} \rho_0 } , \qquad \rho_{n+1} := \rho_n - \delta_n , \qquad \eta_n := \frac{\tM^\alpha}{\gamma}\abs{\bV^{(n)}}_{\rho_n,\Omega_n}^\wlip{\gamma}  , \qquad  N_n:=-\frac{1}{\delta_n}\ln\eta_n $$

\begin{prop}[Iterative Lemma]\label{prop:iter_lemma}
	Fix  $ \tau > 0$.  There exists $\tk_0 \equiv \tk_0(\tau, \delta_0) >0$ such that for any $0 < \gamma < \wtgamma$, any $\tM >0$ for which  
	\begin{equation}\label{eq:iter_0}
		\begin{split}
		\eta_0:=\frac{\tM^\alpha}{\gamma}\abs{\bV^{(0)}}_{\rho_0,\Omega_0}^\wlip{\gamma} \leq \tk_0e^{-1} \ ,
		\end{split}
		\end{equation}
	the following items hold true for any  $n\in\N$:
	\begin{itemize}
	\item[(i)] Setting $\Omega_0$ as in \eqref{eq:magnus_omega},  we have recursively for $n\geq 0$
	\begin{equation*}
		\Omega_{n+1}:=\Set{\omega \in \Omega_{n}\  | \  \abs{\omega\cdot k + \lambda_j^{(n)}(\omega)\pm \lambda_l^{(n)}(\omega)} > \frac{\gamma}{2 \, N_n^{\tau}}\frac{\braket{j\pm l}^{\alpha}}{\tM^{\alpha}} \ , \quad \forall (k,j,l) \in \cI^\pm_{ N_n} } \ ;
	\end{equation*}
	\item[(ii)] For every $\omega \in \Omega_n$, the operator $\bX^{(n)}(\omega, \cdot)\in \lip_\gamma(\Omega_n,\cM_{\rho_{n-1}, s_0}(\alpha,\alpha))$ and 
	\begin{equation}\label{eq:Xn}
		\opnorm{\bX^{(n)}}_{\rho_{n-1},\Omega_n}^{\wlip{\gamma}}\leq \sqrt{\eta_0} \, e^{\frac{1}{2}(1-{\left(\frac{3}{2}\right)}^{n-1})} \  .
	\end{equation}
	\end{itemize}
	The change of coordinates $ e^{\im \bX^{(n)}}$ conjugates $\bH^{(n-1)}$ to  $\bH^{(n)} = \bH_0^{(n)} + \bV^{(n)}$ such that:
	\begin{itemize}
	\item[(iii)] The Hamiltonian   $ \bH_0^{(n)}(\omega)$ is diagonal and time independent, $\bH_0^{(n)}(\omega) = {\rm diag } \{ \lambda_j^{(n)}(\omega)\}_{j \in \N}\bsigma_3$, and the functions $\lambda_j^{(n)}(\omega)=\lambda_j^{(n)}(\omega;\tM,\alpha)$ are defined over all $ \Omega_0$, fulfilling
	\begin{equation}\label{eq:iter_n_l}
		\abs{\lambda_j^{(n)} - \lambda_j^{(n-1)}}^{\rm Lip}_{\Omega_0} \leq \eta_0 \,  e^{1-{\left(\frac{3}{2}\right)}^{n-1}} \ ;
	 \end{equation}
	\item[(iv)]The new perturbation  $\bV^{(n)}\in \lip_\gamma(\Omega_n,\cM_{\rho_n, s_0}(\alpha,0))$ 
	and
	 \begin{equation}\label{eq:iter_n}
		\begin{split}
		\eta_n \equiv  \frac{\tM^\alpha}{\gamma}\abs{\bV^{(n)}}_{\rho_n,\Omega_n}^\wlip{\gamma}  \leq \eta_0 \, e^{1-{\left(\frac{3}{2}\right)}^n} \ .
		\end{split}
	\end{equation}
	\end{itemize}
\end{prop}
\begin{proof}
	We argue by induction. For $n=0$ one requires \eqref{eq:iter_0}.
	Now, assume that the statements hold true  up to a fixed $n\in\N$. 
	Define $\Omega_{n+1}$ as in item $(i)$. 
	In order to apply Lemma \ref{lem:est_gen_kam} and Lemma \ref{lem:est_P+}, we need to check that the assumptions in \eqref{eq:ass_small} and \eqref{eq:small_P} are verified, respectively. First, note that, by item $(iii)$,
	\begin{equation}
		\abs{\lambda_j^{(n)}}^{\rm Lip}_{\Omega_0} \leq \sum_{m=1}^n \abs{\lambda_j^{(m)} - \lambda_j^{(m-1)}}^{\rm Lip}_{\Omega_0}  +\abs{\lambda_j}^{\rm Lip}_{\Omega_0}  \leq \eta_0 \, e \sum_{m=1}^\infty e^{-{\left(\frac{3}{2}\right)}^{m-1}} \leq  \eta_0 \, e , 
	\end{equation}
	 so that \eqref{eq:ass_small} is satisfied, provided  simply $\eta_0 \, e \leq 1$.
	
	We prove now that  \eqref{eq:small_P} is fulfilled. 
	We have
	\begin{align*}
		 \braket{N_n}^{2\tau+1}\eta_n
		\leq \left( \frac{1+n^2}{\delta_0}\right)^{2\tau+1} \, \eta_n^{\frac{1}{2}}		 \stackrel{\eqref{eq:iter_n}}{ \leq} (\eta_0 \, e)^{\frac{1}{2}} \, e^{- \frac{1}{2}{\left(\frac{3}{2}\right)}^n}\left( \frac{1+n^2}{\delta_0}\right)^{2\tau+1}  \leq \frac{1}{2\cdot 16 \cdot C_{s_0}}  
	\end{align*}
	as long as $\eta_0 \, e$ is sufficiently small (depending only on $\delta_0, \tau)$. 
	Therefore we can apply Lemma \ref{lem:est_gen_kam} and Lemma \ref{lem:est_P+} with $\bP \equiv \bV^{(n)}$ and define $\bX^{(n+1)} \in\lip_\gamma(\Omega_{n+1},\cM_{\rho_{n}, s_0}(\alpha,\alpha))$, 
	the new  eigenvalues 
	\begin{equation}\label{new}
		\lambda_j^{(n+1)}(\omega) :=  \lambda_j^{(n)}(\omega) + \wh{(V^{d,(n)})_j^j}(0;\omega) \qquad \forall  \, j\in\N
	\end{equation}
	and the new perturbation $\bV^{(n+1)}$. We are left only with the quantitative estimates.\\
	We start with item $(iv)$. By Remark  \ref{rem:new_pert_est}, one has 
	\begin{equation}\label{eq:quad_kam}
	\eta_{n+1}\leq  \left(1 + \frac{1}{\delta_n^{2\tau+1}}\left(\ln \frac{1}{\eta_n}\right)^{2\tau+1}\right)\eta_n^2 
	\leq 2 \left(\frac{1+n^2}{\delta_0}\right)^{2\tau +1} (\eta_0\, e)^{\frac{7}{4}}e^{-\frac{7}{4}{\left(\frac{3}{2}\right)}^n} \ .
	\end{equation}
	Thus, \eqref{eq:iter_n} is satisfied at the iteration $n+1$ provided again that $\eta_0 \, e$ is sufficiently small (depending only on $\delta_0, \tau)$. 
	For item $(iii)$, it is sufficient to note that
	\begin{equation}
		\abs{\lambda_j^{(n+1)} - \lambda_j^{(n)}}^{\rm Lip}_{\Omega_n} = 	\abs{\wh{(V^{d, (n)})_j^j}(0,\cdot)}^{\rm Lip}_{\Omega_n} \leq \abs{\bV^{(n)}}_{\rho_n,\Omega_n}^\lip \leq \frac{\tM^\alpha}{\gamma}\abs{\bV^{(n)}}_{\rho_n,\Omega_n}^\wlip{\gamma} \stackrel{\eqref{eq:iter_n}}{\leq} \eta_0\,  e^{1- {\left(\frac{3}{2}\right)}^n} .
	\end{equation}
	Now, by Kirszbraun theorem, we can extend the functions $\lambda_j^{(n)}(\omega, \tM)$ to all $\Omega_0$ preserving their Lipschitz constant; this proves $(iii)$.
	Item $(ii)$ is proved in the same lines, using  \eqref{eq:gen_kam_est} and   the inductive assumption; we skip the details.
\end{proof}
A consequence of the iterative lemma is the following result.
\begin{cor}[Final eigenvalues]\label{cor:conv}
Fix $\tau >\wt \tau$ (of Theorem \ref{lem:meln_unpert}). Assume \eqref{eq:iter_0}. Then for every $\omega \in \Omega_0$ and for every $j \in \N$, the sequence $\{ \lambda_j^{(n)}(\cdot; \tM,\alpha)\}_{n \geq 1}$ is a Cauchy sequence. We denote by $\lambda_j^\infty(\omega; \tM,\alpha)$ its limit, which is given by
$
		\lambda_j^{\infty}(\omega)=\lambda_j +\varepsilon_j^{\infty}(\omega)$ 
and one has the estimate
	\begin{equation}
	\label{einf}
	\sup_{j\in\N}\abs{j^\alpha \varepsilon_j^\infty}^{\wlip{\gamma}}_{\Omega_0} \leq \frac{\gamma}{\tM^\alpha}\eta_0 \, e\ .
	\end{equation} 
\end{cor}
\begin{proof}
By \eqref{new} we have $ \varepsilon_j^{\infty}(\omega) := \sum_{n=0}^{\infty}\wh{(V^{(n), d})_j^j}(0,\omega)$. The thesis follows using 
	\begin{equation}
		\abs{j^{\alpha} \widehat{(V^{ (n), d})_j^j}(0;\omega)} \leq \abs{\braket{D}^{\alpha} \wh{V^{(n),d}}(0;\omega)}_{s_0} 
		\leq 
		\abs{\bV^{(n)}}_{\rho_{n},\Omega_{n}}^{\wlip{\gamma}}\stackrel{\eqref{eq:iter_n}}{\leq}\frac{\gamma}{\tM^{\alpha}} \eta_{0} \, e^{1-{\left(\frac{3}{2}\right)}^n} .
	\end{equation}
	\end{proof}

\begin{cor}[Iterated flow]\label{cor:iter_flow}
Fix an arbitrary $r\in[0,s_0]$; 
under the same assumptions of Corollary \ref{cor:conv}, for any $\omega \in \cap_{n} \Omega_n$ and $\theta\in\T^n$, the sequence of transformations
\begin{equation}
\label{seq.Xn}
\cW^n(\theta;\omega) := e^{-\im \bX^{(1)}(\theta;\omega)} \circ \cdots\circ e^{-\im \bX^{(n)}(\theta;\omega)}
\end{equation}
is a Cauchy sequence in  $\cL(\cH^r\times\cH^r)$ fulfilling
\begin{align}
	\norm{\cW^n(\theta;\omega)-\uno}_{\cL(\cH^r\times\cH^r)} \leq \sqrt{\eta_0\, e}\,\Sigma \  e^{\sqrt{\eta_0\, e}\,\Sigma} 
\end{align}
where $\Sigma:=\sum_{q=0}^{\infty}e^{-\frac{1}{2}{\left(\frac{3}{2}\right)}^q}$. 
We denote by $\cW^\infty(\theta;\omega)$ its limit in $\cL(\cH^r\times\cH^r)$.
\end{cor}
\begin{proof}The convergence of the transformations is a  standard argument, while the control of the operator norm
$\cL(\cH^r \times \cH^r)$ follows from Remark
\ref{rem:opnorm_vs_sdecay}; we skip the details.
\end{proof}
Since for any $j \in \N$ the sequence $\{ \lambda_j^{(n)}\}_{n \geq 1}$ converges to a well defined Lipschitz function $\lambda_j^\infty$ defined on $ \Omega_0$, we can now impose second order Melnikov conditions only on the final frequencies.
\begin{lem}[Measure estimates]
\label{lem:meas_infty}
	Consider the set 
	\begin{equation}
	\label{eq:Omegainfty}
	\Omega_{\infty,\alpha} := \left\{ \omega\in \cU_\alpha \  |  \ 
	\abs{\omega \cdot k +\lambda_j^\infty(\omega)  \pm \lambda_l^\infty(\omega) } \geq   \frac{\gamma}{\braket{k}^{\tau}}
	\frac{\braket{j\pm l}^{\alpha}}{\tM^{\alpha}} , 
	  \quad \forall (k,j,l)\in\cI^{\pm}  
 \right\}\ .  
	\end{equation}
	Then 
	 $\Omega_{\infty,\alpha} \subseteq \cap_n \Omega_n$.
	 Furthermore, taking
 $\tau > \nu + \alpha + \frac{\widetilde \tau}{\alpha}$, $\gamma \in [0,  \wtgamma/2]$ and  $\tM \geq \tM_0$ (defined in Theorem \ref{lem:meln_unpert}),   there exists  a constant $C_\infty >0$, independent of $\tM$ and $\gamma$, such that
\begin{equation}
\label{final_est}
\me_r(\cU_\alpha \backslash\Omega_{\infty,\alpha})\leq C_\infty \gamma \ .
\end{equation}
\end{lem}
\begin{proof}
The proof that $\Omega_{\infty,\alpha} \subseteq \cap_n \Omega_n$  is standard,  see e.g. Lemma 7.6 of  \cite{MaPro}.\\
To  prove the measure estimate, 
	let $\omega\in\cU_\alpha$ and $(k,j,l)\in\cI^{\pm}$. 
	We can rule out the  cases as at the beginning of Lemma \ref{lem:new_T2} with essentially the same arguments.
%
%
	Thus, we  restrict to consider all $(k,j,l)\in\cI^{\pm}$ for which $k\neq 0 $ and $j\neq l$.
Furthermore, 	if $\abs{j\pm l}\geq 16\tM\abs k$, we get again that  $\abs{\omega\cdot k + \lambda_j^\infty(\omega)\pm\lambda_l^\infty(\omega)}  \geq  \frac{1}{2}\abs{j\pm l}$
	(recall $\tM > \tm^2$).
	So, we can work in the regions $\abs{j\pm l}<16\tM \abs k$. Now, for $j< l $ satisfying 
	\begin{equation}\label{eq:tilde_R}
		j \braket{j\pm l} \geq  \left( \frac{ 2\eta_0 \, e\braket{k}^{\wt\tau}}{c(\gamma,\wt\gamma)}\right)^{\frac{1}{\alpha}}=:\wt\tR(k) , 
	\end{equation}
	where $c(\gamma,\wt\gamma):=\frac{\wt\gamma}{\gamma}-1>1$ (recall that $\wt\gamma/2>\gamma$), we have (using also \eqref{einf})
	\begin{equation*}
		\begin{split}
		\abs{\omega\cdot k+ \lambda_j^\infty(\omega)\pm\lambda_l^\infty(\omega)} & \geq \abs{\omega\cdot k +\lambda_j \pm \lambda_l} - \abs{\varepsilon_j^\infty(\omega)}-{\abs{\varepsilon_l^\infty(\omega)}}\\
		& \geq \frac{\wtgamma}{\braket{k}^{\widetilde\tau}} \frac{\braket{j\pm l}^{\alpha}}{\tM^{\alpha}}-2\frac{\gamma}{\tM^{\alpha}}\frac{\eta_0 e}{j^{\alpha}}
		\geq \frac{\gamma}{\tM^{\alpha}}\frac{\braket{j\pm l}^{\alpha}}{\braket{k}^{\wt\tau}} \ .
		\end{split}
	\end{equation*}
	Therefore, we can further restrict ourselves to consider just those $j< l$ satisfying $j \braket{j\pm l} < \wt\tR(k)$.
	The symmetric argument leads to work in the sector $j < l$ under the condition  $l \braket{l\pm j} < \wt\tR(k)$.\\
	Now, define the set
	\begin{equation}
		\cG_{j,l}^{k,\pm}:=\Set{ \omega \in R_\tM | \abs{\omega\cdot k+ \lambda_j^\infty(\omega)\pm\lambda_l^\infty(\omega) } < \frac{\gamma}{\braket{k}^{\tau}}\frac{\braket{j\pm l}^{\alpha}}{\tM^{\alpha}} }
	\end{equation}
	for those $k\neq 0$ and $j\neq l$ in the region
	\begin{equation}\label{eq:restr_infty}
		\cR^\pm:=\{ \abs{j\pm l} < 16\tM\abs k \} \cap \left( \{ j\braket{j\pm l}< \wt\tR(k)  ,\ j<l  \} \cup \{ l\braket{j\pm l} < \wt\tR(k), \ l<j \} \right) ;
	\end{equation}
	Recall that  $f_{kjl}^\pm(\omega):=\omega\cdot k+\lambda_j^\infty(\omega)\pm\lambda_l^\infty(\omega)$ are Lipschitz functions on $R_\tM$. For $k \neq 0$, since
	 $\abs{\lambda^\infty_l}^\lip_{R_\tM}  < |k|/4$, by  
	Lemma \ref{tecnico} we get
	\begin{equation*}
		\abs{\cG_{j,l}^{k,\pm}} \lesssim \tM^{\nu-\alpha}\gamma\frac{\braket{j\pm l}^{\alpha}}{\abs k^{\tau+1}} \ .
	\end{equation*}
	Define $\cG_{\infty}^\pm:= \bigcup\Set{ \cG_{j,l}^{k,\pm} | (k,j,l)\in\cR^\pm  } \cap \cU_\alpha $. We have
	\begin{equation*}
	\begin{split}
	\abs{\cG_{\infty}^-} &  \lesssim  2\gamma\, \tM^{\nu-\alpha} \sum_{k\neq 0} \sum_{j< l \atop j\braket{j-l}< \wt\tR(k)}\sum_{\abs{j-l}< 16\tM\abs k} \frac{\braket{j-l}^{\alpha}}{\abs k^{\tau+1}}
	 \lesssim  \gamma\, \tM^{\nu-\alpha} \sum_{k\neq 0} \ \ \sum_{l-j=:h>0 \atop\abs{h}< 16\tM\abs k} \ \ 
	 \sum_{ j< \wt\tR(k)\braket{h}^{-1}} \frac{\braket{h}^{\alpha}}{\abs k^{\tau+1}}\\
	&\lesssim \frac{\gamma}{c(\gamma,\wt\gamma)^{\frac{1}{\alpha}}}\tM^\nu\sum_{k\neq 0} \frac{1}{\abs k^{\tau+1-\alpha-\frac{\wt\tau}{\alpha}}} 
	 \lesssim \frac{\gamma}{c(\gamma,\wt\gamma)^{\frac{1}{\alpha}}}\tM^\nu \lesssim \gamma\, \tM^\nu \ ,
	\end{split}
	\end{equation*}
%
%
	taking $\tau+1-\alpha-\frac{\wt\tau}{\alpha}> \nu$. The same computation holds for $\cG_{\infty}^+$, and proves \eqref{final_est}. 
\end{proof}
\begin{thm}[KAM reducibility]\label{thm_kam_red}
	Fix $\alpha \in (0,1)$,  $s_0 > 1/2$, and $\tau > \nu + 1 + \alpha +\frac{\wt\tau}{\alpha}$. For any  $0 < \gamma < \wtgamma$, there exists $\tM_*=\tM_*(\tm, \alpha,\gamma,\rho_0) >0 $ such that for any $\tM \geq \tM_*$ the following holds true.
	There exist functions $\{\lambda_j^{\infty}(\omega; \tM,\alpha)\}_{j \in \N}$, defined and Lipschitz in $\omega$ in the set $R_\tM$ such that:
	\begin{itemize}
	\item[(i)] The set $\Omega_{\infty,\alpha}=\Omega_{\infty,\alpha}(\gamma, \tau, \tM)\subset R_\tM $ defined  in \eqref{eq:Omegainfty} fulfills $\me_r(R_\tM \setminus \Omega_\infty) \leq C (\gamma +\wtgamma^{1/3}+ \gamma_0)$,
	where $\gamma_0$ is defined in Theorem \ref{lem:magnus} and $\wtgamma$ in Theorem \ref{lem:meln_unpert}.
	\item[(ii)] For each $\omega \in \Omega_{\infty,\alpha}$ there exists a change of coordinates $\psi = \cW^\infty(\omega t, \omega) \phi$ which conjugates  equation \eqref{eq:system_kg_pert}
	to a constant-coefficient diagonal one:
	\begin{equation}
	\label{diag_eq}
		\im \dot \phi=\bH^{\infty}\phi \ , \quad \bH^{\infty}=\bH^{\infty}(\omega,\alpha)=\diag\{ \lambda_j^{\infty}(\omega,\alpha) \ | \ j\in\N \}\bsigma_3 \ .
	\end{equation}
	Furthermore for any $r\in[0,s_0]$ one has
	\begin{equation}
		\norm{\cW^\infty - \uno}_{\cL(\cH^r \times \cH^r)} \leq \sqrt{\eta_0\, e}\,\Sigma \  e^{\sqrt{\eta_0\, e}\,\Sigma} \ .
	\end{equation}
	\end{itemize}
\end{thm}
\begin{proof}
	Having fixed $\alpha, s_0$ and $\tau$, we can produce the constant ${\rm k}_0(\delta_0, \tau)$ of the iterative Lemma \ref{prop:iter_lemma}. 
	Having fixed also $0<\gamma < \wtgamma$, we produce $\tM_*>0$ in such a way that for every $\tM \geq \tM_*$, the estimate \eqref{eq:iter_0} is fulfilled. 
	We can now apply the iterative Lemma \ref{prop:iter_lemma}, Corollary \ref{cor:conv} and Lemma \ref{lem:meas_infty} to get the result.
\end{proof}

\subsection{A final remark}
The KAM reducibility scheme that we have presented has transformed Equation \eqref{eq:system_kg_pert} into \eqref{diag_eq}, 
where the asymptotic for the final eigenvalues are given, using Equation \eqref{einf}, by
\begin{equation}
	\lambda_j^\infty(\omega,\alpha) - \lambda_j \sim O\left( \frac{\eta_0}{\tM^\alpha j^\alpha} \right) \stackrel{\eqref{eq:step0_kam}}{\sim} O\left( \frac{1}{\tM j^\alpha} \right) \ .
\end{equation}
One can argue that the asymptotic $\lambda_j^{\infty}(\alpha)- \lambda_j \sim O(\tM^{-1}j^{-\alpha}) $ is not that satisfying, since the pertubation $\bV^{(0)}$ at the beginning of the KAM scheme belongs to the class $\cM_{\rho_0, s_0}(1,0)$ and so its diagonal elements have a smoothing effect of order 1 which could be expected to be preserved in the effective Hamiltonian.\\
Actually, it is possible to modify our reducibility scheme for achieving this result: we explain now briefly how to do it. 
After the Magnus normal form, we conjugate system \eqref{eq:system_kg_pert} through $e^{-\im \bY(\omega t)}$, where
\begin{equation}
	\bY(\omega t):= \left( \begin{matrix}
	0 & Y^o(\omega t) \\ - \bar{Y^o}(\omega t) & 0 
	\end{matrix} \right)
\end{equation} 
so that $Y^o$ solves the homological equation
\begin{equation}
	-\im [Y^o(\theta),B]_{\rm a} + V^o(\theta)- \omega\cdot \partial_{\theta}Y^o(\theta) = 0 \ \Rightarrow \ \wh{(Y^o)_j^l}(k) := \frac{\wh{(V^o)_j^l}(k)}{\im(\omega\cdot k + \lambda_j + \lambda_l)} \quad \forall \, k,j,l \ .
\end{equation}
We ask now the frequency vector $\omega$ to belong to $\cU_1 \cap \cU_0$ (see \eqref{eq:U_alpha}). In this way  one gets (in the same lines of the proof of Lemma \ref{lem:est_gen_kam}) that $\bY \in  \lip_{\gamma/\tM}(\cU_1,\cM_{\wt\rho_0, s_0}(1,1))$, since  we have chosen $\omega\in\cU_1$, with the bound
\begin{equation}
\label{Y}
	\abs{\bY}_{\wt\rho_0,s_0,1,1}^{\wlip{\gamma/\tM}}\leq C \frac{\tM}{\gamma}\abs{\bV^{(0)}}_{\rho_0,s_0,1,1}^\wlip{\gamma/\tM} \leq C \frac{\tM}{\gamma}\abs{\bV^{(0)}}_{\rho_0,s_0,1,0}^\wlip{\gamma/\tM} \stackrel{\eqref{eq:magnus_size}}{\leq} \wt{C} \ .
\end{equation}
The new perturbation 
\begin{equation}
	\wt{\bV^{(0)}}(\omega t):= \left(\begin{matrix}
	V^d(\omega t) & 0 \\ 0 & - \bar{V^d}(\omega t)
	\end{matrix}\right) + \int_0^1 (1-s)\LieTr{s\bY(\omega t)}{\ad_{\bY(\omega t)}[\bV^{(0)}(\omega t)]}\wrt s
\end{equation}
belongs to the class $\lip_{\gamma/\tM}(\cU_1,\cM_{\wt\rho_0, s_0}(1,1))$ fulfilling estimate \eqref{eq:magnus_size}.\\
Thus, one can perform a  KAM reducibility scheme  as in Section \ref{subsec:est_gen_kam}--\ref{sub:iter_kam}, in which one takes $\alpha = 0$ in  \eqref{eq:meln_set}, the perturbations appearing in the iterations stay  in the class $\lip_{\gamma/\tM^0}(\wt{\Omega_n},\cM_{\wt\rho_n, s_0}(1,1))$ and the new final eigenvalues $\wt{\lambda_j^\infty}$ satisfy the nonresonance condition
\begin{equation}
	\abs{\omega\cdot k + \wt{\lambda_j^\infty}\pm \wt{\lambda_l^\infty} } \geq \frac{\gamma}{\braket{k}^\tau} ,  \quad \ \forall \, (k,j,l)\in\cI^{\pm} \ .
\end{equation}
In particular, we obtain better asymptotics on the final eigenvalues, that is $\wt{\lambda_j^\infty} - \lambda_j \sim O(\tM^{-1}j^{-1})$. The price that we pay for this result is that the preliminary change of coordinate $e^{-\im\bY(\omega t)}$ is not a transformation close to identity, as the generator $\bY(\omega t)$ is just a bounded operator and not small in size, see \eqref{Y}.
The main consequence is that the effective dynamics of the original system, as Corollary \ref{cor.1} is no more valid. In this case, it is possible to conclude just that the Sobolev norms stay uniformly bounded in time and do not grow, but in general their (almost-)conservation is lost.

\appendix

\section{Technical results}
\label{app:tr}

\subsection{Properties of pseudodifferential operators}
Recall that if $F$ is an operator, we denote by $\wh F(k)$ its $k^{th}$ Fourier coefficient defined  as in \eqref{norm:rs}.
If $F$ is a pseudodifferential operator with symbol $f$, so   $\wh F(k)$ is, with symbol given by
$$
\wh f(k,x, j) := \frac{1}{(2\pi)^\nu} \int_{\T^\nu} f(\theta, x, j) \, e^{- \im \theta \cdot k} \, \di \theta .
$$

\begin{lem}
\label{lem:fou.symb}
Let $\rho >0$ and $\mu \in \R$. The following holds true:
\begin{itemize}
\item[(i)]  If  $F \in \cA^\mu_\rho$, then the operator   $\wh F(k)$ belongs to $ \cA^\mu$ for any $k \in \Z^\nu$ and 
$$
\wp^\mu_\ell(\wh F(k)) \leq  e^{-\rho  |k|} \, \wp^{\mu,\rho}_\ell(F) \ \quad \forall \ell \in \N_0 \ .
$$ 
\item[(ii)] Assume to have $\forall k \in \Z^\nu$  an operator   $\wh F(k) \in \cA^\mu$ fulfilling 
\begin{equation}
\label{ipotesi}
\wp^\mu_\ell(\wh F(k)) \leq \la k \ra^\tau \,  e^{-\rho |k|} \, C_\ell  \ \quad  \forall k \in \Z^\nu , \ \ \forall \ell \in \N_0 \ ,
\end{equation}
for some $\tau \geq 0$,  $\rho >0$ and $C_\ell>0$ independent of $k$. 
Define the operator $F(\theta) := \sum_{k \in \Z^\nu} \wh F(k) e^{\im \theta\cdot k}$.
Then, $F$ belongs to $\cA^\mu_{\rho'}$  for any $0<\rho' <\rho$  and one has 
$$
\wp^{\mu,\rho'}_\ell(F) \leq  \frac{C_\ell }{(\rho-\rho')^{\tau+\nu}} \ \quad \forall \ell \in \N_0 \ .
$$
\end{itemize}
On the classes $\lip_\tw(\Omega,\cP\cA_\rho^\mu)$, these assertions extend naturally without any further loss of analyticity.
\end{lem}
\begin{proof}
	(i) By Cauchy estimates, it is well-known the analytic decay for the Fourier coefficients of the symbol $f(\theta;x,j)$:
	\begin{equation}
		\abs{\wh f(k,x,j)}\leq
		e^{-\rho \abs k}\sup_{\abs{{\rm Im}\theta}\leq  \rho}\abs{f(\theta,x,j)} \ .
	\end{equation} 
	Plugging it into Definition \ref{defn:symbols_semi} of $\wp_{\ell}^{\mu}(\wh F(k))$, we get the claim;\\
	(ii) It is possible to control the seminorm $\wp_{\ell}^{\mu,\rho'}(F)$ in terms of the ones for the Fourier coefficients:
	\begin{equation}
		\wp_{\ell}^{\mu,\rho'}(F)\leq \sum_{k \in \Z^{\nu}}e^{\rho'\abs k}\wp_{\ell}^\mu(\wh F(k)) 
		\stackrel{\eqref{ipotesi}}{\leq} 
		\sum_{k \in \Z^\nu} e^{(\rho' -\rho)|k|} \la k \ra^\tau C_l \leq \frac{C_l }{(\rho-\rho')^{\tau+\nu}} 
		\ .
	\end{equation}
\end{proof}

In the next Proposition we essentially prove that pseudodifferential operators as in Definition \ref{def:pseudo_lip}  have  matrices which belong to the classes $\lip_\tw(\Omega,\cM_{\rho,s})$  extended from Definition \ref{def:mat}.
\begin{prop}
\label{prop:optoma}
Let $F \in \lip_\tw(\Omega,\cP\cA^\mu_\rho)$, with $\rho >0$.  
 For any $0 < \rho' < \rho$ and $s >\frac{1}{2}$,  the matrix of the operator
$$
\la D \ra^{\alpha} F \,  \la D\ra^{\beta}  , \ \quad   \alpha + \beta  + \mu \leq 0 \ , 
$$
belongs to $\lip_\tw(\Omega,\cM_{\rho', s})$.
Moreover for any $s >\frac{1}{2}$, $\forall \alpha+ \beta \leq - \mu$,  there exists $\sigma >0$ such that
\begin{equation}
\label{est:optoma}
\abs{\la D \ra^{\alpha} F \,  \la D\ra^{\beta}}_{\rho', s,\Omega}^\wlip{\tw} \leq \frac{C}{(\rho - \rho')^\nu}  \ \wp^{\mu, \rho}_{s+\sigma}(F)_\Omega^\wlip{\tw} . 
\end{equation}
\end{prop}
\begin{proof}
Since $\braket{D}\in\cP\cA^1$ is clearly independent of parameters, without loss of generality let $F$ belong to $\cP\cA_\rho^\mu $. We start by proving the result in the case  $\mu = \alpha = \beta = 0$. Let an arbitrary $s>\frac{1}{2}$ be fixed. 
Then
\begin{equation}
	\begin{split}
	\wh F_m^n(k) & :=  \frac{1}{(2\pi)^\nu}\int_{\T^\nu\times[0,\pi]}F(\theta,x,D_x)[\sin(mx)]\sin(nx)e^{-\im k\cdot \theta} \wrt \theta \wrt x \\
	 & = \frac{1}{2(2\pi)^\nu}\int_{\T^\nu\times[-\pi,\pi]}F(\theta,x,D_x)[\sin(mx)]\sin(nx)e^{-\im k\cdot \theta} \wrt \theta \wrt x\\
	& =  \frac{1}{4(2\pi)^\nu}\int_{\T^{\nu+1}}f(\theta,x,m)(e^{\im(m-n)x}-e^{\im(m+n)x})e^{-\im k\cdot \theta} \wrt \theta \wrt x \ ,
	\end{split}
\end{equation}
where $f\in \cP S_\rho^m$ is the symbol of $F$.
Consider first the case $m \neq n$. Then, integrating by parts $\wt s$-times in $x$, with $\wt s:=\lfloor s+2 \rfloor +1 $, and shifting the contour of integration in $\theta$ to $\T^\nu - \im \rho\,\sgn(k)$ (here $\sgn(k):=(\sgn(k_1),...,\sgn(k_\nu))\in\{-1,1\}^\nu$),  one gets that for any $n, m \in \N$, $n\neq m$, $ k \in \Z^\nu$ 
\begin{align*}
\abs{\wh F_m^n(k)} & \leq e^{-\rho |k|} \left(\frac{1}{|m+n|^{\wt s}} + \frac{1}{|m-n|^{\wt s}}\right) \,
\sup_{|{\rm Im \theta }| < \rho  \atop (x, m) \in \T \times \N} \abs{\partial_x^{\wt s} f(\theta; x, m)} 
  \leq  \frac{ 2  e^{-\rho |k|}}{|m-n|^{\wt s}}\,  \wp^{0, \rho}_{\wt s}(f) \ .
\end{align*}
If $m=n$, in a similar way one proves the bound 
$\sup_{m \in \N} \abs{\wh F_m^m(k)} \leq  e^{-\rho |k|}\,  \wp^{0, \rho}_{0}(f)$.
It follows  that for any $ 0<\rho' < \rho$, one has 
$\abs{F}_{\rho', s} \leq C (\rho-\rho')^{-\nu}\wp^{0, \rho}_{\wt s}(f)< \infty$, which proves \eqref{est:optoma} in the case $\alpha = \beta = \mu =0$.
To treat the general case, it is sufficient to note that,  by  Remarks \ref{rem:comp}, \ref{DPA} and  \ref{rem:comp_pp}, the operator $\la D \ra^{\alpha} F \,  \la D\ra^{\beta} \in \cP\cA^0_\rho$, so we have
\begin{equation}
\abs{\la D \ra^{\alpha} F \,  \la D\ra^{\beta}}_{\rho', s} \leq \frac{C}{(\rho - \rho')^\nu}  \ \wp^{0, \rho}_{s+\sigma}(\la D \ra^{\alpha} F \,  \la D\ra^{\beta}) \leq \frac{C_{\alpha, \beta}}{(\rho - \rho')^\nu}  \ \wp^{\mu, \rho}_{s+\sigma}( F ) . 
\end{equation}
\end{proof}

\subsection{Proof of Lemma \ref{lem:emb} (Embedding)}
The result now follows immediately by applying Proposition \ref{prop:optoma} to $F\in\lip_\tw(\Omega,\cP\cA_\rho^{-\alpha})$ and $G\in\lip_\tw(\Omega,\cP\cA_\rho^{-\beta})$. Indeed, we obtain
$$ \abs{\braket{D}^{\sigma}F\braket{D}^{-\sigma}}_{\rho',s,\Omega}^\wlip{\tw},
 \ \ \  \ 
\abs{\braket{D}^\alpha F}_{\rho',s,\Omega}^\wlip{\tw},
\ \ \ \ 
\abs{F\braket{D}^\alpha}_{\rho',s,\Omega}^\wlip{\tw}\leq \frac{C}{(\rho-\rho')^\nu}\wp_{s+\sigma}^{-\alpha,\rho}(F)_\Omega^\wlip{\tw} \ . $$
The estimates for $G$ are analogous.

\subsection{Proof of Lemma \ref{rem:alg}}
\label{alg.s.decay}
Denote by $A_e$ the extension of the operator $A$ on $L^2(\T)$ which coincides with $A$ on $L^2_{odd}(\T) \equiv \cH^0$ and  is identically zero on $L^2_{even}(\T)$.
Since $A_e$ is parity preserving, one verifies for any $m, n \in \Z$ that  $\la A_e \, e^{\im m x}, e^{\im n x} \ra_{L^2(\T)} = 2\la A \sin(m x), \sin(n x)\ra_{\cH^0}$ .
Therefore, \eqref{eq:algebra_sdecay} is equivalent to the classical algebra property developed on the exponential basis (for instance, see \cite{bebo13}); we skip the details.

\subsection{Proof of Lemma \ref{lem:com}(Commutator)}
\label{proof:lem:com}
We start with operators independent of $\theta\in\T^\nu$. Let 
$$
\bX = 	\left(\begin{matrix}
			X^d&X^o\\-\overline{X^o}&-\overline{X^{d}}
			\end{matrix}\right) \ , 
			\qquad  \bV = \left(\begin{matrix}
			V^d&V^o\\-\bar{V^o} & -\bar{V^d}
			\end{matrix}\right) \ .
$$
One has
\begin{equation*}
	\im [\bX,\bV] = \im\left( \bX\bV-\bV\bX \right) =  \left( \begin{matrix}
	\im Z^d & \im Z^o \\ -\wo{(\im Z^o)} & -\wo{(\im Z^d)}
	\end{matrix} \right) \ ,
\end{equation*}
where
\begin{align*}
	Z^d  := X^dV^d-X^o\wo{V^o} - V^dX^d+V^o\wo{X^o} , \qquad  \ \ \ 
	Z^o  := X^dV^o-X^o\wo{V^d} - V^dX^o+V^o\wo{X^d} \ . 
\end{align*}
Omitting sake of simplicity conjugate operators and labels for diagonal and anti-diagonal elements, by Remark \ref{rem:ana}, the following inequalities hold (here $\sigma=\pm \alpha,0$):
\begin{equation}
\begin{split}
\abs{\braket{D}^{\sigma}XV\braket{D}^{-\sigma}}_s 
& \leq C_s \abs{\braket{D}^{\sigma}X\braket{D}^{-\sigma}}_s\abs{\braket{D}^{\sigma}V\braket{D}^{-\sigma}}_s \ ; \\
\abs{\braket{D}^{\alpha}XV}_s & \leq C_s \abs{\braket{D}^{\alpha}X}_s\abs{V}_s \ ;\\
\abs{XV\braket{D}^{\alpha}}_s 
& \leq C_s \abs{X\braket{D}^{\alpha}}_s\abs{\braket{D}^{-\alpha}V\braket{D}^{\alpha}}_s \ ;
\end{split}
\end{equation}
the same for those terms involving $VX$.
All these norms extend easily to the analytic case. Therefore, by the assumption and from the definition in \eqref{eq:sdecay_matrix_norm}, properties \ref{M1}, \ref{M2} and \ref{M3} are satisfied. It remains to show the symmetries conditions in \eqref{struttura}. Note that
$
	(\im Z^d)^* = \im Z^d$ and $ (\im Z^o)^*=\wo{\im Z^o} $ if and only if  $ (Z^d)^*=-Z^d , \ (Z^o)^*=\wo{Z^o}$.
We check the condition for $Z^d$. We have
\begin{equation}
	\begin{split}
	(Z^d)^* &  =  (V^d)^*(X^d)^*-(\wo{V^o})^*(X^o)^*-(X^d)^*(V^d)^*+(\wo{X^o})^*(V^o)^* \\
	& = V^d X^d - V^o \wo{X^o} - X^d V^d + X^o \wo{V^o}^* = - Z^d \ .
	\end{split}
\end{equation}
In the same way one checks that $(Z^o)^*=\wo{Z^o}$.
The Lipschitz dependence is easily checked.

\small
 
\newcommand{\etalchar}[1]{$^{#1}$}
\def\cprime{$'$}

\end{document}